\documentclass{article}
\usepackage[dvipsnames]{xcolor}
\usepackage{gastex}
\usepackage{amsmath}
\usepackage{amssymb}
\usepackage{upref}
\usepackage{multirow}
\usepackage{multicol}
\usepackage{url}
\usepackage{graphicx}
\usepackage{tikz}
\usepackage{subfigure}
\usepackage{graphicx}

\usepackage{makeidx}
\makeindex
\usepackage{theorem}
\newtheorem{theorem}{Theorem}
\newtheorem{lemma}[theorem]{Lemma}
\newtheorem{proposition}[theorem]{Proposition}
\newtheorem{corollary}[theorem]{Corollary}
\newtheorem{definition}[theorem]{Definition}
{\theorembodyfont{\rmfamily}%
  \newtheorem{example}[theorem]{Example}
  \newtheorem{remark}[theorem]{Remark} }
\newenvironment{proof}{\noindent\textit{Proof.}}
{\QED\vskip\theorempostskipamount} 
\newenvironment{proofof}[1]{\noindent\textit{Proof
    \protect{#1}.}}
                       {\QED\vskip\theorempostskipamount}
\def\petitcarre{\vrule height4pt width 4pt depth0pt}
\def\QED{\relax\ifmmode\eqno{\hbox{\petitcarre}}\else{%
  \unskip\nobreak\hfil\penalty50\hskip2em\hbox{}\nobreak\hfil
  \petitcarre
  \parfillskip=0pt \finalhyphendemerits=0\par\smallskip}
  \fi}
\newcommand\RR{\mathcal{R}}
\newcommand\CR{\mathcal{CR}}

\newcommand{\N}{\mathbb{N}}

\newcommand{\R}{\mathbb{R}}
\newcommand{\CC}{\mathbb{C}}
\newcommand\MR{\mathcal{MR}}

\def\un(#1){\underline{#1}\,}
\DeclareMathOperator{\Card}{Card}

\DeclareMathOperator{\Fact}{Fact}

\definecolor{ivoire}{rgb}{0.99,0.99,0.8}
\def\FFF{\mathcal{F}}

\def\smc{symmetric basis }
\def\epsilon{\varepsilon}
\usepackage{calc}
\newcounter{hours}\newcounter{minutes}
\newcommand\computetime{\setcounter{hours}{\time/60}%
  \setcounter{minutes}{\time-\value{hours}*60}%
  \thehours\,h\,\theminutes}
\newcommand\dateandtime{\today\quad\computetime}
\usepackage[hypertex,hyperindex,pagebackref,final]{hyperref}
\numberwithin{theorem}{section}
\numberwithin{equation}{section}
\numberwithin{figure}{section}
\numberwithin{table}{section}
\title{Return words   of linear involutions and fundamental groups }
\author{Val\'erie Berth\'e$^1$, Vincent Delecroix$^2$,
Francesco Dolce$^3$, \\
Dominique Perrin$^3$,
Christophe  Reutenauer$^4$,
Giuseppina Rindone$^3$\\\\
$^1$CNRS, Universit\'e Paris 7,
$^2$CNRS, Universit\'e de Bordeaux,\\
$^3$Universit\'e Paris Est, $^4$Universit\'e du Qu\'ebec \`a Montr\'eal}
\date{\dateandtime}
\begin{document}
%========================
%========== Lists ==================================
\makeatletter
\def\@listI{%
  \leftmargin\leftmargini
  \setlength{\parsep}{0pt plus 1pt minus 1pt}
  \setlength{\topsep}{2pt plus 1pt minus 1pt}
  \setlength{\itemsep}{0pt}
}
\let\@listi\@listI
\@listi
\def\@listii {%
  \leftmargin\leftmarginii
  \labelwidth\leftmarginii
  \advance\labelwidth-\labelsep
  \setlength{\topsep}{0pt plus 1pt minus 1pt}
}
\def\@listiii{%
  \leftmargin\leftmarginiii
  \labelwidth\leftmarginiii
  \advance\labelwidth-\labelsep
  \setlength{\topsep}{0pt plus 1pt minus 1pt}
%              \topsep    0\p@ \@plus\p@\@minus\p@
  \setlength{\parsep}{0pt} 
  \setlength{\partopsep}{1pt plus 0pt minus 1pt}
}
\makeatother
%====================

\newcommand{\comment}[1]{ \begin{center} {\fbox{\begin{minipage}[h]{0.9
            \linewidth} {\sf #1}  \end{minipage} }} \end{center}} 
\maketitle
%========================

\begin{abstract}
We investigate the natural codings of linear involutions. We deduce
from the  geometric
representation  of linear involutions   as Poincar\'e maps of measured foliations
   a suitable  definition
of return words which  yields  that  the set of first return
words to a given word is a symmetric basis of the free group on the underlying
alphabet $A$.    The set of   first return words    with respect to   a subgroup of
finite index $G$ of the free group on $A$ is  also  proved to be a symmetric basis of $G$.
 \end{abstract}
%\tableofcontents
%%%%%%%%%%%%%%%%%%%%%

\section{Introduction}
A linear involution is an injective   piecewise isometry defined on a  pair of intervals.  This   generalization of   the notion of  interval
exchange
allows one   to work with  nonorientable foliations on nonorientable surfaces.
Linear involutions  were introduced  by  Danthony and Nogueira
in~\cite{DanthonyNogueira1990} and~\cite{DanthonyNogueira1988},  generalizing  
interval exchanges  with flip(s) \cite{Nogueira1989,NogueiraPiresTroubetzkoy2013} (these are interval exchange transformations which reverse orientation 
 in at least one interval). 
They  extended  to these transformations the notion
of Rauzy induction (introduced in~\cite{Rauzy1979}). The study of linear involutions
was later developed by Boissy and Lanneau in~\cite{BoissyLanneau2009}.
Note that there exist   various  generalizations of interval exchanges:    let us  quote, e.g.,   pseudogroups   of isometries \cite{GaboriauLevittPaulin1995} and    interval identification systems \cite{Skripchenko2012}.

In the present paper, we  study natural codings of linear
involutions  in the spirit of our previous papers on Sturmian sets
\cite{BerstelDeFelicePerrinReutenauerRindone2012} 
and their generalizations as tree sets \cite{BertheDeFeliceDolceLeroyPerrinReutenauerRindone2013c,BertheDeFeliceDolceLeroyPerrinReutenauerRindone2013b,BertheDeFeliceDolceLeroyPerrinReutenauerRindone2013m,BertheDeFeliceDolceLeroyPerrinReutenauerRindone2013a}.
A tree set  is  a  factorial set of words   that  all satisfy a combinatorial 
condition   expressed in terms of the possible extensions  of these  words within  the   tree set: 
 the condition is  that  the extension graph of each word is a tree, with this graph describing  the possible extensions of a word in the language on the left and on the right. 
Tree sets encompass  the languages of   classical shifts of zero entropy like the ones generated by  Sturmian words, Arnoux-Rauzy words,  or else  natural codings of interval exchanges. Note however  that these shifts display various  behaviors in terms  of spectral properties (they can  be weakly mixing, or  they can have pure discrete spectrum).

Tree sets  have particularly interesting  properties  relating    free groups,  symbolic dynamics and bifix codes. In particular tree sets allow  one to  exhibit  bases
of the free group,  or  of subgroups  of the free group.  Indeed,   in a uniformly recurrent tree set, the sets of first return words  to a given word are bases of the free group on the alphabet
 \cite{BertheDeFeliceDolceLeroyPerrinReutenauerRindone2013a}. Moreover, maximal  bifix codes  that are included in
uniformly recurrent  tree sets provide bases of subgroups of finite index  of the free group  \cite{BertheDeFeliceDolceLeroyPerrinReutenauerRindone2013b}.
 Tree sets are also proved to be  closed under maximal bifix decoding and under decoding with respect to
 return words \cite{BertheDeFeliceDolceLeroyPerrinReutenauerRindone2013m}.

 All these properties thus hold  for  regular interval exchange sets. Observe that the fact   that  first return words are bases of the free group  can either be deduced  combinatorially  from the property that  interval exchanges yield
tree sets  \cite{BertheDeFeliceDolceLeroyPerrinReutenauerRindone2013c} or else, as we will show here,  from the  geometric interpretation  of interval exchanges as Poincar\'e sections of linear
flows on translation surfaces:  for any   
word $w$ of  the associated  language,  the set of return words  
to  $w$    provides  a   basis of the fundamental group  of the associated surface.

The natural coding of a linear involution  is the set  of factors  of  the  infinite words that  encode the sequences of subintervals met by the orbits of the transformation.
They are defined  on an
alphabet $A$  whose letters and their inverses
index the intervals exchanged by the involution.    A natural coding  is thus a subset of the free group $F_A$ on the alphabet $A$.  An important property of this set is 
its stability by taking inverses.

 We extend to natural codings of linear involutions most of the properties
proved for uniformly recurrent tree sets, and thus, for   natural codings of interval exchanges.  The  extension is not   completely immediate.  If linear involutions   have a geometric interpretation  as Poincar\'e maps of measured foliations,
%see Section \ref{sectionGeometric} for more details), 
one has   to modify the definition
of return words in order to  make it consistent with  the notion of Poincar\'e  map  of  a foliation. We  thus consider return words
 to the set $\{w, w^{-1}\}$ and we consider a truncated version of them, that we call
 mixed first  return words.   We also have  to replace the basis of a subgroup by
its symmetric version containing the inverses of its elements,
called a symmetric  basis.  The free group is then    obtained as the fundamental group
of   a compact surface  in which a finite number of points   are removed, and linear involutions are  seen  as 
 Poincar\'e sections of  measured foliations  of  the surface. The   return words to a given word can be seen as different ways of choosing a section.

We prove that if $S$ is the
natural coding of a linear involution $T$ without connection on the alphabet $A$, 
the following holds.
\begin{itemize}
\item The set
of mixed  first return words to a given word in $ S$
is a symmetric basis of the free group on $A$
(Theorem~\ref{theoremReturns}).%or First Return Theorem).
%\item A finite symmetric bifix code $X$ is $S$-maximal if and only if it is a symmetric
%basis of a subgroup of finite index (Theorem~\ref{theoremFiniteIndex} or Finite Index Basis Theorem).
\item  
Let $G$ be a subgroup  of finite index of the free group $F_A$. The set of  prime words in $S$ with respect to $G$
is a symmetric basis of $G$ (Theorem \ref{theoremGroupCode}).  
By prime words in $S$ with respect to $G$, we mean the
nonempty words
in $G \cap S$  without a proper nonempty prefix  in $G \cap S$.
\end{itemize}

Observe that return words play  a  crucial role in   symbolic dynamics. They allow
the characterization of   substitutive words \cite{Durand:98},  they provide spectral information  through eigenvalues
(see, e.g.,  \cite{BS:2014}),  or else, they  yield   so-called $S$-adic representations  \cite{Durand:00,Durand:03}.

Let us stress the fact that   even if the proofs provided here  concerning the   algebraic properties
 of   return words  are of a  topological and geometric flavor,  these properties  hold in a wider combinatorial context
  through 
the  notion  of specular set   and specular  groups, 
 where the present   geometric   background  does a priori
 not exist.  Specular  groups   are natural generalizations of free groups: 
they  are free products of a finite number of copies of  ${\mathbb Z}$  and ${\mathbb  Z}/2{\mathbb Z}$.  
A specular set is a subset of  a specular  group which generalizes the
 natural codings of linear involutions. More precisely, we consider an alphabet  with an involution $\theta$ acting on $A$,  possibly with  some fixed points,
 and the group $G_{\theta}$  generated by $A$ with the  relations $a \theta(a)=1$  for every letter $a$ in $A$.  We  can thus  consider,  in this extended framework,
 reduced words, symmetric sets of words  as  well as laminary sets.  In the case 
 where $\theta$ has no fixed point, we recover the free group.
 A specular set is then  defined  as    a laminary set  such that  the extension graph of any nonempty word is a tree and the extension graph of the  
 empty word has two connected components which are trees. 
Extensions of  Theorem \ref{theoremReturns} and \ref{theoremGroupCode}  are proved to  hold in this  context in  \cite{Specular2015}.
%;   special sets also allow  a characterization of the symmetric bases of subgroups of finite index of
 %specular groups contained in a specular set $S$: those are      the finite $S$-maximal symmetric
% bifix codes contained in $S$ \cite{Specular2015}.   

\bigskip

This paper is organized as follows.   In Section~\ref{sectionPreliminaries}, we recall notions concerning
words, free groups and graphs. Linear involutions are defined in 
Section~\ref{sectionInvolutions}.
 We    also recall that, by a result of \cite{BoissyLanneau2009}, a  nonorientable linear involution without connection
is minimal.  In  Section~\ref{sectionGeometric}, we  provide the  necessary  geometric background 
on   natural involutions.   We focus on the    symbolical properties  of their natural codings   in Section~\ref{sec:returnFund}
and we   introduce  the notion of  return word we will work with, as well   as even  letters and the   even group.
%In Section~\ref{sectionNatural},  we introduce the natural coding of a linear involution.
  The geometric  and topological proofs of the main results
on return words for  natural codings of linear involutions are given in Section \ref{sec:return}.

\paragraph{Ackowledgement} This work was supported by grants from
Region Ile-de-France (project  DIM RDM-IdF),    and by the ANR projects Dyna3S ANR-13-BS02-0003  and Eqinocs ANR-11-BS02-004.
We warmly thank the referees for their    valuable  comments.
%, and  the FARB Project
%``Aspetti algebrici e computazionali nella teoria dei codici,
%degli automi e dei linguaggi formali'' (University of Salerno, 2013).
%and the MIUR PRIN 2010-2011 grant
%``Automata and Formal Languages: Mathematical and Applicative Aspects''.

%The authors thank the referee for his suggestions which helped
%to improve the presentation of our paper.

%%%%%%%%%%%%%%%%%%%%%%
\section{Words, free groups and laminary sets}\label{sectionPreliminaries}
In this section, we introduce notions concerning  sets of words and free groups.

%\subsection{Recurrent and uniformly recurrent sets}
Let $A$ be a finite nonempty alphabet and let $A^*$ be  the set of all words on $A$.
We  let $1$ or  $\varepsilon$ denote  the empty word.
A set of words is said to be \emph{factorial} if it contains the
factors of its elements.

The notation
$a^{-1}$ will be interpreted  
as an inverse in the free group   $F_A$ on $A$. We also use
the notation $\bar{a}$ instead of $a^{-1}$. 
%Let $A^{-1}=\{a^{-1}\mid a\in A\}$
%be a copy of $A$. The map $a\mapsto a^{-1}$ is extended to an involution
%on $A\cup A^{-1}$ by defining $(a^{-1})^{-1}=a$. The notation
%$a^{-1}$ will be interpreted  
%as an inverse in the free group   $F_A$ on $A$. We also use
%the notation $\bar{a}$ instead of $a^{-1}$.
%Recall that  the  elements of $F_A$  are identified
%with the reduced words on the alphabet $A\cup A^{-1}$, with 
%a word on $A\cup A^{-1}$ being  \emph{reduced} if it
%has
%no factor of the form $aa^{-1}$ for $a\in A\cup A^{-1}$. Note that the set
%of reduced words is factorial. For any word $w$ on $A\cup A^{-1}$ there
%is a unique reduced word equivalent to $w$ obtained by erasing
%he factors $aa^{-1}$
%or $a\in A\cup A^{-1}$. If $u$ is the reduced word equivalent to $w$, we say that
%$w$ \emph{reduces} to $u$. If $w=a_1a_2\cdots a_n$ with $a_i\in A\cup A^{-1}$
%is reduced, the inverse of $w$ is the word $w^{-1}=a_n^{-1}\cdots a_2^{-1}a_1^{-1}$.
%All words considered below,
%unless
%stated explicitly, are supposed to be on the alphabet $A\cup A^{-1}$.

A set of reduced words on the alphabet $A\cup A^{-1}$ is said to be
\emph{symmetric} if it contains the inverses of its elements.
Let  $X^*$  be the submonoid 
of $(A\cup A^{-1})^*$ generated
by $X$ without reducing the products. If $X$ is symmetric, the subgroup  of $F_A$  generated by   $X$
is the set obtained by reducing the words of $X^*$.

\begin{definition}[Symmetric basis] If $X$ is a basis of a subgroup $H$ of $F_A$, the set $X\cup X^{-1}$
is  called a \emph{symmetric basis}
of $H$. 
\end{definition}
In particular, $A\cup A^{-1}$ is a symmetric  basis of $F_A$. Note that
a symmetric  basis $X\cup X^{-1} $  is not a basis of $H$ but that any $w\in H$
can be written uniquely $w=x_1x_2\cdots x_n$ with $x_i\in X\cup X^{-1}$
and $x_ix_{i+1}$ not equivalent to $1$ for $1\le i\le n-1$.  We recall that, by \emph{Scheier's Formula}, any basis of a subgroup
of index $d$ of a free group on $k$ symbols has $d(k-1)+1$ elements.
Hence, if $Y$ is a symmetric  basis
of a subgroup of index $d$ in a free group on $k$ symbols, then $\Card(Y)=2d(k-1)+2$.

\bigskip

The following definition follows~\cite{HilionCoulboisLustig2008}
and~\cite{LopezNarbel2013}.
\begin{definition}[Laminary set]\label{def:laminary}
A symmetric factorial set of reduced words on the alphabet $A\cup A^{-1}$
is called
a \emph{laminary set} on $A$.
\end{definition}

A  laminary set $S$ is called \emph{semi-recurrent} if for any $u,w\in S$,
there is a $v\in S$ such that $uvw\in S$ or $uvw^{-1}\in S$. Likewise,
it is said to be \emph{uniformly semi-recurrent} if it
is right extendable and
 if for any word $u\in S$ there is an integer $n\ge 1$ such that
for any word $w$ of length $n$ in $S$, $u$ or $u^{-1}$ is a factor of $w$.
A uniformly semi-recurrent set is semi-recurrent.

Following again the terminology of~\cite{HilionCoulboisLustig2008}, we say
that a laminary set $S$ is \emph{orientable} if there exist two factorial
sets $S_+,S_-$ such that $S=S_+\cup S_-$ with $S_+\cap S_-=\{\varepsilon\}$
and for any $x\in S$, one has $x\in S_-$ if and only if $x^{-1}\in S_+$.
Note that if $S$ is a semi-recurrent orientable laminary set, then the sets $S_+,S_-$ as above
are unique (up to their interchange). The sets $S_+,S_-$ are called the
\emph{components} of $S$.
Moreover a uniformly recurrent and orientable laminary set is
a union of two uniformly recurrent sets.
Indeed,  $S_+$ and $S_-$ are uniformly recurrent.

%%%%%%%%%%%%%%%%%%%
\section{Linear involutions}\label{sectionInvolutions}
In this section, we define linear involutions, which are a generalization
of interval exchange transformations. We first give the basic definitions
including generalized permutation and length data, and
then  discuss minimality for  involutions  in relation with  the notion of connection.
%%%%%
\subsection{Definition}
Let $A$ be an alphabet with $k$ elements. 

We consider two copies $I\times \{0\}$ and
$I\times \{1\}$ of an open interval $I$ of the real line and  we define
 $\hat{I}=I\times \{0,1\}$.
We call the sets $I\times \{0\}$ and $I\times \{1\}$ the two
\emph{components} of $\hat{I}$. We consider each component as an open interval.

A \emph{generalized permutation} on $A$ of type $(\ell,m)$, with $\ell+m=2k$,
  is a bijection $\pi:\{1,2,\ldots,2k\}\rightarrow A\cup
A^{-1}$.
We represent it by a two line array
\begin{displaymath}
\pi=\begin{pmatrix} \pi(1)\ \pi(2)\ \ldots \pi(\ell)\\
\pi(\ell +1)\ \ldots \pi(\ell+m)
\end{pmatrix}\ .
\end{displaymath}
A \emph{length data} associated with $(\ell,m,\pi)$ is a nonnegative
vector $\lambda\in \R_+^{A\cup A^{-1}}=\R_+^{2k}$ such that
\begin{displaymath}
\lambda_{\pi(1)}+\ldots+\lambda_{\pi(\ell)}=
\lambda_{\pi(\ell+1)}+\ldots+\lambda_{\pi(2k)}
\text{ and }\lambda_a=\lambda_{a^{-1}}\text{ for all }a\in A.
\end{displaymath}
We consider a partition of $I\times \{0\}$ (minus $\ell-1$
points) into $\ell$ open intervals
$I_{\pi(1)},\ldots,I_{\pi(\ell)}$ of lengths $\lambda_{\pi(1)},\ldots,\lambda_{\pi(\ell)}$
and a partition of $I\times \{1\}$ (minus $m-1$ points) into $m$ open intervals
$I_{\pi(\ell+1)},\ldots,I_{\pi(\ell+m)}$ of lengths $\lambda_{\pi(\ell+1)},\ldots,\lambda_{\pi(\ell+m)}$. Let $\Sigma$ be the set of $2k-2$ \emph{division points} separating
the intervals $I_a$ for $a\in A\cup A^{-1}$.

The \emph{linear involution} on $I$ relative to these data is the
map $T=\sigma_2\circ\sigma_1$ defined on the set
$\hat{I}\setminus\Sigma$, formed of $\hat{I}$ minus $2k-2$ points,
and
which is
 the composition
of two involutions defined as follows. 
\begin{enumerate}
\item[(i)]The first involution $\sigma_1$ is defined on $\hat{I}\setminus\Sigma$.
It is such that for each $a\in A\cup A^{-1}$, its restriction to $I_a$
is either a translation or a symmetry from $I_a$ onto $I_{a^{-1}}$.
 Thus, there are real
numbers
$\alpha_a$ such that for any $(x,\delta)\in I_a$,
one has $\sigma_1(x,\delta)=(x+\alpha_a,\gamma)$ in the first case,
and $\sigma_1(x,\delta)=(-x+\alpha_a,\gamma)$ in the second case
(with $\gamma\in\{0,1\}$).
\item[(ii)]The second involution exchanges the two components of
  $\hat{I}$.
It  is defined for $(x,\delta)\in \hat{I}$
by $\sigma_2(x,\delta)=(x,1-\delta)$. The image of $z$ by $\sigma_2$
is called the \emph{mirror image} of $z$.
\end{enumerate}
We also say that $T$ is a {\em linear involution on $I$  relative to
 the alphabet $A$}
or that it is a {\em $k$-linear involution} to express the fact
that the alphabet $A$ has $k$ elements.

\begin{example}\label{exampleLinear}
Let $A=\{a,b,c,d\}$ and
\begin{displaymath}
\pi=\begin{pmatrix}a&b&a^{-1}&c\\c^{-1}&d^{-1}&b^{-1}&d
\end{pmatrix} \ .
\end{displaymath}
Let $T$ be the $4$-linear involution corresponding to the length data
represented in Figure~\ref{figureLinear}. We represent $I\times\{0\}$
above $I\times \{1\}$ with the assumption that the restriction
of $\sigma_1$ to $I_a$ and $I_d$ is a symmetry while its restriction
to $I_b,I_c$ is a translation.
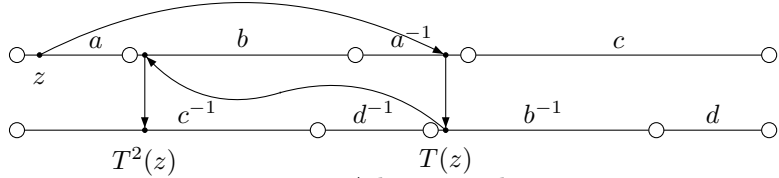
\begin{figure}[hbt]
\centering
\gasset{Nadjust=wh,AHnb=0}
\begin{picture}(100,17)(0,-1)
%--interval nodes--
\node(h0)(0,10){}\node(h1)(15,10){}\node(h2)(45,10){}\node(h3)(60,10){}\node(h4)(100,10){}
\node(b0)(0,0){}\node(b1)(40,0){}\node(b2)(55,0){}\node(b3)(85,0){}\node(b4)(100,0){}
\gasset{Nh=.1,Nw=.1,Nadjust=n}
%iteration nodes
\node[ExtNL=y,Nh=.6,Nw=.6,Nfill=y,NLangle=-90,NLdist=2](z)(3,10){$z$}\node[Nh=.6,Nw=.6,Nfill=y](sz)(57,10){}\node[ExtNL=y,Nh=.6,Nw=.6,Nfill=y,NLangle=-90,NLdist=2](Tz)(57,0){$T(z)$}
\node[Nh=.6,Nw=.6,Nfill=y](sTz)(17,10){}\node[ExtNL=y,Nh=.6,Nw=.6,Nfill=y,NLangle=-90,NLdist=2](TTz)(17,0){$T^2(z)$}
\node(m)(35,5){}
%interval lines
\drawedge[ELpos=70](h0,h1){$a$}\drawedge(h1,h2){$b$}\drawedge(h2,h3){$a^{-1}$}\drawedge(h3,h4){$c$}
\drawedge[ELpos=60](b0,b1){$c^{-1}$}\drawedge(b1,b2){$d^{-1}$}\drawedge(b2,b3){$b^{-1}$}\drawedge(b3,b4){$d$}
%iteration edges
\gasset{AHnb=1}
\drawedge[AHnb=1,curvedepth=7](z,sz){}\drawedge[AHnb=1](sz,Tz){}
\drawedge[curvedepth=-3,AHnb=0](Tz,m){}\drawedge[curvedepth=3](m,sTz){}\drawedge(sTz,TTz){}
\end{picture}
\caption{A linear involution.}\label{figureLinear}
\end{figure}
We indicate on the figure the effect of the transformation $T$ on a point
$z$ located in the left part of the interval $I_a$. The point
$\sigma_1(z)$ is located in the right part of $I_{a^{-1}}$,  and the point
$T(z)=\sigma_2\sigma_1(z)$ is just below on the left of $I_{b^{-1}}$.
 Next, the point $\sigma_1T(z)$ is located on the left part of $I_b$
and the point $T^2(z)$ just below.
\end{example}

Thus the notion of linear involution is an extension of the notion 
of  interval exchange transformation in the following sense.
Assume that $\ell=k$, that
$A=\{\pi(1),\ldots,$ $\pi(k)\}$,  and that the restriction of $\sigma_1$
to each subinterval is a translation. Then, the restriction of $T$
to $I\times \{0\}$ is an interval exchange (and so is its restriction to 
$I\times \{1\}$ which is the inverse of the first one). Thus,
in this case, $T$ is a pair of mutually inverse interval exchange transformations.

It is also an extension of the notion of interval exchange with flip(s)
\cite{Nogueira1989,NogueiraPiresTroubetzkoy2013}. Assume again  that $\ell=k$, that
$A=\{\pi(1),\ldots,$ $\pi(k)\}$, but now that the restriction of $\sigma_1$
to at least one  subinterval is a symmetry. Then the restriction of $T$
to $I\times \{0\}$ is an interval exchange  with flip(s).

Note that we consider in this paper interval exchange transformations
defined by a partition of an open interval  minus
$\ell-1$ points in $\ell$
 open intervals. The usual notion of interval exchange transformation
uses a partition of a semi-interval in a finite number of semi-intervals.
One recovers the usual notion of interval exchange transformation
 on a semi-interval
 by attaching to each open interval its left endpoint.

A linear involution $T$ is a bijection from $\hat{I}\setminus\Sigma$
onto $\hat{I}\setminus\sigma_2(\Sigma)$.
Since $\sigma_1,\sigma_2$ are involutions and $T=\sigma_2\circ\sigma_1$,
 the inverse of $T$
is $T^{-1}=\sigma_1\circ\sigma_2$.

The set $\Sigma$ of division points is also the set of singular points
of $T$ and their mirror images are the singular points of $T^{-1}$.
Note that these singular points $z$ may be `false' singularities, in the sense
that $T$ can have a continuous extension to an open neighborhood of $z$.

Two particular cases of linear involutions deserve attention.
\begin{definition}[Nonorientable linear involution]
A linear involution $T$ on the alphabet $A$ 
relative to a generalized permutation $\pi$ of type
$(\ell,m)$ 
is said to be \emph{nonorientable} if there are indices $i,j\le \ell$ such that
$\pi(i)=\pi(j)^{-1}$  (and thus indices $i,j\ge \ell+1$
such that $\pi(i)=\pi(j)^{-1}$).  In other words,  there is  some $a\in A\cup A^{-1}$ for which  $I_a$
and $I_{a^{-1}}$ belong to  the same component  of $\hat{I}$. Otherwise
$T$ is said to be \emph{orientable}.  \end{definition}

\begin{definition}[Coherent linear involution]
A linear involution $T=\sigma_2\circ \sigma_1$ on $I$
relative to the alphabet $A$
is said to be \emph{coherent} if, for each $a\in A\cup A^{-1}$, the restriction
of $\sigma_1$ to $I_a$ is a translation if and only if $I_a$
and $I_{a^{-1}}$ belong to distinct components of $\hat{I}$. 
\end{definition}
\begin{example}
The linear involution of Example~\ref{exampleLinear} is coherent.
%\begin{example}[A noncoherent  linear involution] \label{exampleNonCoherent}
Let us consider now  the linear involution $T$ which is the same as in Example~\ref{exampleLinear}, but such that the restriction of $\sigma_1$ to $I_c$ 
is a symmetry. Thus $T$ is not coherent. We assume that $I=]0,1[$,
that $\lambda_a=\lambda _d$, that $1/4<\lambda_c<1/2$
 and that $\lambda_a+\lambda_b<1/2$.
 Let $z=1/2+\lambda_c$ (see Figure~\ref{figureLinearnonCoherent}).
We have then $T^3(z)=z$, showing that $T$ is not minimal. Indeed,
since $z\in I_c$, we have
$T(z)=1-z=1/2-\lambda_c$. Since
$T(z)\in I_a$ we have $T^2(z)=(\lambda_a+\lambda_b)+(\lambda_a-1+z)=z-\lambda_c=1/2$.
Finally, since $T^2(z)\in I_{d^{-1}}$, we obtain $1-T^3(z)=T^2(z)-\lambda_c=1-z$
and thus $T^3(z)=z$.
\begin{figure}[hbt]
\centering
\gasset{Nadjust=wh,AHnb=0}
\begin{picture}(100,18)(0,-2)
%--interval nodes--
\node(h0)(0,10){}\node(h1)(15,10){}\node(h2)(45,10){}\node(h3)(60,10){}\node(h4)(100,10){}
\node(b0)(0,0){}\node(b1)(40,0){}\node(b2)(55,0){}\node(b3)(85,0){}\node(b4)(100,0){}
\gasset{Nh=.6,Nw=.6,Nfill=y,Nadjust=n,ExtNL=y}
%iteration nodes
\node[NLdist=2](z)(90,10){$z=T^3(z)$}\node[Nh=.1,Nw=.1](m)(50,5){}\node(sz)(10,0){}\node[NLangle=-140,NLdist=2](Tz)(10,10){$T(z)$}\node(sTz)(50,10){}\node[ExtNL=y,NLangle=-90,NLdist=2](TTz)(50,0){$T^2(z)$}\node(sTTz)(90,0){}
%interval lines
\drawedge[ELpos=70](h0,h1){$a$}\drawedge(h1,h2){$b$}\drawedge(h2,h3){$a^{-1}$}\drawedge(h3,h4){$c$}
\drawedge[ELpos=60](b0,b1){$c^{-1}$}\drawedge(b1,b2){$d^{-1}$}\drawedge(b2,b3){$b^{-1}$}\drawedge(b3,b4){$d$}
%iteration edges
\gasset{AHnb=1}
\drawedge[curvedepth=3,AHnb=0](z,m){}\drawedge[curvedepth=-3](m,sz){}
\drawedge(sz,Tz){}\drawedge[curvedepth=4](Tz,sTz){}\drawedge(sTz,TTz){}
\drawedge[curvedepth=-4](TTz,sTTz){}\drawedge(sTTz,z){}
\end{picture}
\caption{A noncoherent linear involution.}\label{figureLinearnonCoherent}
\end{figure}
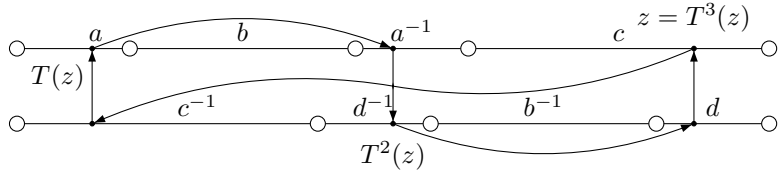
\end{example}

Linear
involutions which are orientable and coherent
correspond to interval exchange transformations, 
whereas  orientable but noncoherent  linear
involutions are interval exchanges with flip(s).

Orientable linear involutions correspond to orientable laminations (see
Section~\ref{sectionGeometric}), whereas 
 coherent
linear involutions correspond to orientable surfaces. Thus coherent
nonorientable involutions correspond to nonorientable laminations
on orientable surfaces.

%%%%%%%%%%%%%%%%%%%%%
\subsection{Minimality}\label{subsec:minimal}
 We first  recall the notion of  connection  and then 
prove 
 that involutions without connection are essentially
always minimal.

%We say that a transformation $T$ defined on a topological space
%$X$ is \emph{minimal} if the
%nonnegative orbit
%$P(z)=\cup_{n\ge 0}T^n(z)$ of any point  $z\in X$
%is dense in $X$. 

\begin{definition}[Connection]
A \emph{connection} of a linear involution $T$ is a triple $(x,y,n)$ 
where $x$ is a singularity of $T^{-1}$, $y$ is a singularity of $T$, $n\geq 0$  and $T^n x = y$. 
\end{definition}
%\comment{We call $n$ the length of the connection.}

Let $T$ be a  linear involution without connection. Let 
\begin{equation}
O=\bigcup_{n\ge 0}T^{-n}(\Sigma) \quad \text{and }\quad \hat{O}=O\cup \sigma_2(O)\label{eqO}
\end{equation}
be respectively the negative orbit of the singular points and its
closure under mirror image.
Then $T$ is a bijection from $\hat{I}\setminus \hat{O}$ onto itself.
Indeed, assume that $T(z)\in \hat{O}$. If $T(z)\in O$ then
$z\in \hat{O}$. Next if $T(z)\in \sigma_2(O)$, then
$T(z)\in \sigma_2(T^{-n}(\Sigma))=T^n(\sigma_2(\Sigma))$ for some $n\ge 0$. We cannot
have $n=0$ since $\sigma_2(\Sigma)$ is not in the image of $T$.
Thus $z\in T^{n-1}(\sigma_2(\Sigma))=\sigma_2(T^{-n+1}(\Sigma))\subset
\sigma_2(O)$. Therefore in both cases $z\in \hat{O}$. The converse
implication
is proved in the same way. Note  that $\hat{I}\setminus\hat{O}$ is dense in $\hat{I}$, 
and the   nonnegative orbit of any  point of  $\hat{I}\setminus\hat{O}$   is   well-defined.  

 \begin{definition}[Minimality]
 A linear involution $T$ 
on $I$ without connection
is  minimal if for any point $z\in\hat{I}\setminus\hat{O}$ 
the nonnegative orbit of $z$  is dense in $\hat{I}$. 
\end{definition}

Note that when a linear involution is  orientable, that is,  when it is a pair
of interval exchange transformations (with or without flips),   the interval exchange transformations
can be minimal although the linear involution is not since each component
of $\hat{I}$ is stable by the action of $T$. 
Moreover, it is shown in~\cite{DanthonyNogueira1990} that noncoherent linear 
involutions are almost surely not minimal.

Let $X \subset I \times \{0,1\}$. The \emph{return time} $\rho_X$ to $X$ is the function from $I \times \{0,1\}$ to $\N \cup \{\infty\}$ defined on $X$  by
\[
\rho_X(x) = \inf \{n \geq 1 \mid T^n(x) \in X\}.
\]

The following result is proved in~\cite{BoissyLanneau2009}
(Proposition 4.2) for the class of coherent involutions.
The proof uses Keane's theorem proving that an
interval exchange transformation without connection is minimal~\cite{Keane1975}. The proof of Keane's theorem also implies that for each interval of positive length, the return
time to this interval is bounded.
\begin{proposition}\label{propBL}
Let $T$ be a linear involution without connection on $I$. If $T$
is nonorientable, it is minimal. Otherwise, its restriction to each component
of $\hat{I}$ is minimal. Moreover, for each interval of positive length
included in $\hat{I}$, the return time to this interval takes
a finite number of values.
\end{proposition}
\begin{proof}
Consider the set $\widetilde{I}=\hat{I}\times\{0,1\}=I\times\{0,1\}^2$ and the transformation
$\widetilde{T}$ on $\widetilde{I}$ defined for $(x,\delta)\in\widetilde{I}$ by
\begin{displaymath}
\widetilde{T}(x,\delta)=\begin{cases}
(T(x),\delta)&\text{if $T$ is a translation on a neighborhood of $x$}\\
(T(x),1-\delta)&\text{otherwise.}
\end{cases}
\end{displaymath}

Let $T'$ be the transformation induced by $\widetilde{T}$ on $I'=I\times\{0,0\}$.
Note that if $x\in I'$ is recurrent, that is, $\widetilde{T}^n(x)\in I'$ for some
$n>0$, then the restriction of $T'$ to some neighborhood of $x$
is a translation. Indeed, there is an even number of indices $i$ with
$0\le i< n$ such that $T$ is a symmetry on a neighborhood of $T^i(x)$.

Let us show that $T'$ is an interval exchange transformation.
Let $\Sigma$ be the set of singularities of $T$.
For each $z\in\Sigma$, let $s(z)$ be the minimal integer $s>0$
(or $\infty$) such that $\widetilde{T}^{-s}(z)\in I'$. Let $N=\{\widetilde{T}^{-s(z)}(z)\mid z\in \Sigma\text{ with } s(z)<\infty\}$. The set $N$ divides $I'$ into  a finite number of disjoint
open intervals. If $J$ is such an open interval, it contains, by the
Poincar\'e Recurrence Theorem, at least one recurrent point $x\in I'$
 for $\widetilde{T}$,
that is such that $\widetilde{T}^n(x)\in I'$ for some $n>0$. By definition of $N$, all the points
of $J$ are recurrent. Moreover, as we have seen above, the restriction of $T'$ to $J$ is a translation. This shows that $T'$ is an interval exchange
transformation.

We can now conclude the proof. Since $T$ has no connection, $T'$ has no connection. Thus, by Keane's theorem, it is minimal. This shows that 
the intersection with $I\times \{0\}$ of the nonnegative
orbit of any point in $I\times \{0\}$ is dense in $I\times \{0\}$. A similar
proof shows that the same is true for $I\times \{1\}$. If $T$
is nonorientable, the nonnegative orbit of any $x\in I\times \{0\}$ contains
a point in $I\times \{1\}$. Thus its nonnegative orbit is dense in $\hat{I}$.
The same holds symmetrically for $x\in I\times\{1\}$.

Let $J$ be an interval of positive length included in $I$.
By Keane's theorem, the return time to $J\times\{0,0\}$ relative to $T'$ takes
a finite number of values. Thus
the return time to $J\times\{0\}$
with respect to $T$ takes also a finite number of values. A similar
argument holds for an interval included in $I\times\{1\}$.
\end{proof}

%%%%%%%%%%%%%%%%%%%%%%%%%
\section{Measured foliations   and linear involutions} \label{sectionGeometric}
In order to study  return words of linear involutions  (this will be the object of Section \ref{sec:returnFund} and \ref{sec:return}),
we first  introduce a  geometric and topological  viewpoint on natural involutions.
The main actors are measured foliations of surfaces introduced by W.P. Thurston~(see \cite{FathiLaudenbachPoenaru1991} for an introduction,  and see also~\cite{Hatcher2002}). They can be  considered
 as two-dimensional extensions of  linear involutions. 
They are defined on a compact surface $X$ in which a finite number of points $\Sigma \subset X$ are removed.
Poincar\'e sections of these measured foliations are then  linear involutions.

 A foliation is a decomposition of a surface as a union of leaves which are 1-dimensional. As an example, the plane $\R^2$ decomposes as a union of vertical lines.
Let $X$ be a (non-necessarily orientable) surface. A \emph{foliation} on $X$ is a covering of $X$ by charts $\phi_i: X_i \rightarrow \R^2$
such that the transitions $\phi_i \circ \phi_j^{-1}: \phi_j(X_i \cap X_j) \rightarrow \phi_i(X_i \cap X_j)$ preserve
vertical lines, in other words they are of the form:
\[
\phi_i \circ \phi_j^{-1}(x,y) = (f_{ij}(x), g_{ij}(x,y))
\]
with $f_{ij}(x)=\pm x + c_{ij}$. In the chart $\phi_j$, each stripe $x = a$ matches up with the stripe $x = f_{ij}(a)$ in $X_i$. Gluing all together
these stripes we obtain a \emph{leaf} of the foliation which is a one-dimensional manifold  immersed  in $X$. Each leaf is hence homeomorphic to the circle $\mathbb{R}/\mathbb{Z}$ or the line $\mathbb{R}$. The surface $X$ decomposes as the union of these leaves.

Given a nonsingular smooth vector field, or more generally a line field, the integral curves of this field provide a foliation.
\begin{example}\label{exampleInvolution3}
Let $T$ be the coherent linear involution on $I=]0,1[$
represented in Figure~\ref{figureLinear3}.
We choose $(3-\sqrt{5})/2$ for the length of the interval $I_c$
(or $I_b$). With this choice, $T$ has no connection.
\begin{figure}[hbt]
\centering
\gasset{AHnb=0,Nadjust=wh}
\begin{picture}(100,20)
\node(h0)(0,10){}\node(b)(23.6,10){}\node(bbar)(61.8,10){}\node(h1)(100,10){}
\node(b0)(0,0){}\node(cbar)(38.2,0){}\node(abar)(76.4,0){}\node(b1)(100,0){}

\drawedge(h0,b){$a$}\drawedge(b,bbar){$b$}\drawedge(bbar,h1){$b^{-1}$}
\drawedge(b0,cbar){$c$}\drawedge(cbar,abar){$c^{-1}$}\drawedge(abar,b1){$a^{-1}$}
\end{picture}
\caption{The linear involution on  a tree-letter alphabet  of Example \ref{exampleInvolution3}.}\label{figureLinear3}
\end{figure}
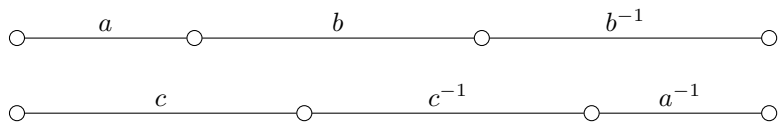

In Figure~\ref{fig:suspension} we show an example of a foliation of a surface related to this linear involution. This surface is built from a polygon where vertices are removed and edges are glued with orientation preserving isometry.
\begin{figure}[hbt]
\begin{center}\includegraphics{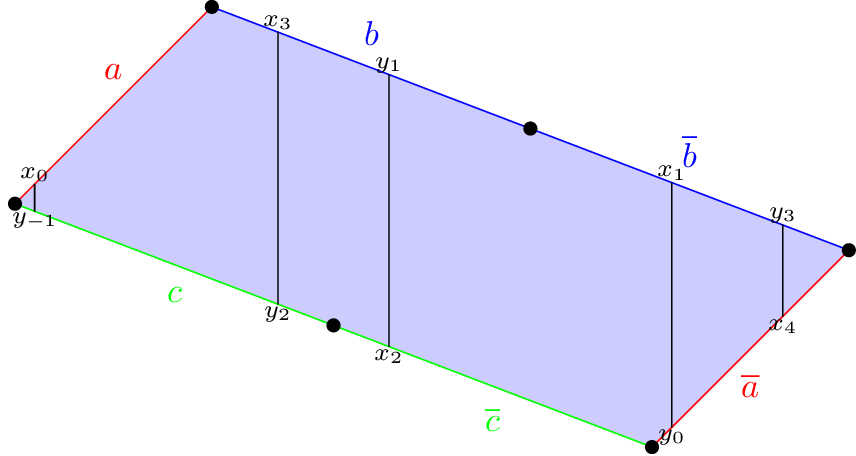} 
%\begin{center}\includegraphics{suspension.pdf} 
\end{center}
\caption{A suspension of a linear involution whose vertical lines naturally form a foliation. The cutting sequence following a leaf is given by the iteration of the linear involution (the notation follows the convention $\sigma_1(y_i) = x_i$ and $\sigma_2(y_i) = x_{i+1}$).}
\label{fig:suspension}
\end{figure}
\end{example}

Let now $X$ be a compact surface.
A \emph{singular foliation} on $X$ is a foliation $\FFF$ defined on $X \backslash \Sigma$ where $\Sigma \subset X$ is a finite set of points and such that in the neighborhood of each point of $\Sigma$ the foliation is homeomorphic to the foliation of the punctured disc in $\CC$ given by the line field $z^p (dz)^2 = I$; in other words, the leaves are the branches of $\gamma_c(t) = (I t + c)^{1/(p/2+1)}$ where $c \in \mathbb{C}$ is a constant (see also Figure~\ref{fig:singular_charts} for a picture). In this foliation there are  $p+2$ singular leaves (which are half-lines that hit $0$) that we call \emph{separatrices}. We say that the singularity of the foliation has \emph{angle} $(p+2) \pi$ or \emph{degree} $p$.

\begin{figure}[!ht]
\begin{center}
\subfigure[degree $p=-1$, angle  $\pi$\label{subfig:LABEL1}]{
\includegraphics{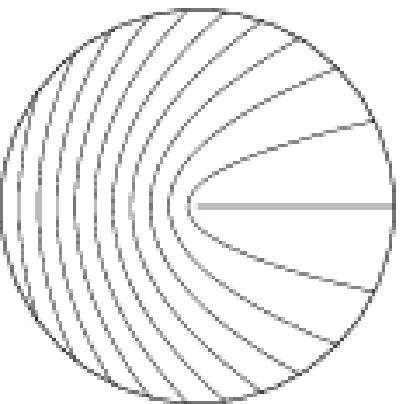}
} \hspace{1cm}
\subfigure[degree $p=1$, angle  $3\pi$\label{subfig:LABEL2}]{
\includegraphics{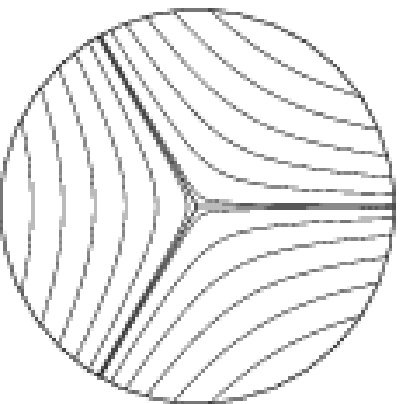}
} \hspace{1cm}
\subfigure[degree $p=2$, angle  $4\pi$\label{subfig:LABEL3}]{
\includegraphics{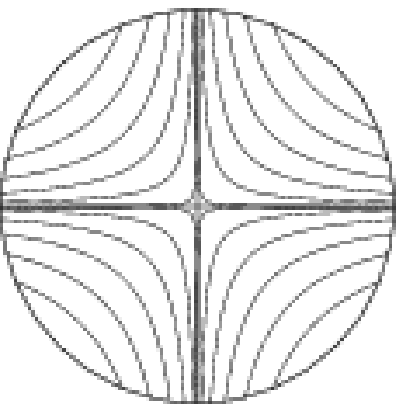}
}
\end{center}
\caption{Chart around points of $\Sigma$.}
\label{fig:singular_charts}
\end{figure}

On the surface obtained from the polygon of Figure~\ref{fig:suspension}, 
one can check that the foliation has 4 singularities of degree $p=-1$ (or angle $\pi$).

A \emph{transverse measure} on $\FFF$ is a measure $\mu$ defined on transverse arcs to $\FFF$ that is invariant under homotopy along the leaves and which is finite on compact intervals. A \emph{measured foliation} is a singular foliation endowed with a transverse measure.  We will see that linear involutions and measured foliations are essentially the same objects. In Figure~\ref{fig:suspension}, the natural transverse measure is simply the integral of $dx$ along curves (where $x$ is the natural horizontal coordinate in the plane).

A measured foliation is denoted as  $(X,\Sigma,\FFF,\mu)$ or $(\FFF,\mu)$ when the space $X$ and the set $\Sigma$ are understood.

A \emph{connection} of $\FFF$ is a finite leaf that joins two points of $\Sigma$.

\begin{definition}\label{definition1}
Let $(X,\Sigma,\FFF,\mu)$ be a measured foliation without connection.
A closed segment $I \subset X$ is admissible if
\begin{itemize}
\item it is transverse to $\FFF$,
\item its interior avoids $\Sigma$ and both endpoints are on singular leaves,
\item the leaf segments that join one endpoint to a singularity do not intersect the interior of $I$.
\end{itemize}
\end{definition}
We consider admissible intervals as being oriented, that is,
 having a start and an end.
Because of the transverse measure, there is always  a prefered parametrization for segments: we always assume that parametrization of a segment $\gamma: [0,t] \rightarrow X$ is such that $\mu(\gamma([s,s'])) = s' - s$. In other words, there is a unique parametrization such that $\mu|_I$ is the image of the Lebesgue measure. For a transverse segment $I$ and $\delta > 0$ small enough, there is a neighborhood of $I$ which is isomorphic to $[0,\mu(I)] \times [-\delta,\delta]$, and for which the leaves of the foliation on the rectangle are the vertical segments. For a piece of leaf that crosses the segment $I$, it hence makes sense to say \emph{going up} or \emph{going down}.

We define the {\em Poincar\'e map} of the foliation on $I \times \{0,1\}$ as follows. 
For a point $x \in I$, we define $\sigma_1(x,0)$ as the point $(y,i) \in I \times \{0,1\}$ where $y$ is the first point of the interior of $I$ that is crossed
 by following the leaf from $x$ and going up.
If we arrive from above we set $i=0$ and if not we set $i=1$. 
Next, $\sigma_1(x,1)$ is defined similarly, but following the leaf from $x$ 
by going down. The map $\sigma_1$ is not defined if the leaf encounters
 a singularity before returning into $I$. The map $\sigma_2$ is the exchange $(x,0) \mapsto (x,1)$ and $(x,1) \mapsto (x,0)$. The transformation
$T$ is the composition $\sigma_2 \circ \sigma_1$. 
The sequence $(x,0)$, $T(x,0)$, $T^2(x,0), \ldots$ 
is by construction the sequence of intersections of the leaf from $x$ with $I$.
Note that the  way the  Poincar\'e map of the foliation  works    explains    the notion of   mixed first  return word (see Definition \ref{def:mixed} below).

The \emph{total angle} of a foliation is the sum of
the angles of the singularities.
\begin{lemma} \label{lem:Poincare_map}
Let $(X,\Sigma,\FFF,\mu)$ be a measured foliation without  connection of total angle $(2k-2) \pi$. Let $I$ be an admissible interval. Then the Poincar\'e map induced on $I \times \{0,1\}$ is a $k$-linear involution  without connection.
\end{lemma}

\begin{proof}
If the foliation has no connection, then each infinite half-leaf intersects $I$.
We consider singularities for the Poincar\'e map, in other words the points in $I$ that run into a singularity before going back to $I$.
This set cuts the domain $I \times \{0,1\}$ into subintervals. 
As the transverse measure is preserved, the Poincar\'e map is an isometry restricted to each of these subintervals. It is hence a linear involution.

For each subinterval $I_a$, let $R_a$ be the rectangle made of the union of the leaf segments that start from $I_a$ to $T(I_a)$. On each of the two vertical boundaries of these rectangles there is exactly one singularity except for two of the extreme rectangles. It follows that there are $k$ pairs of subintervals for the Poincar\'e map.
\end{proof}
Note that if $p_1,\ldots,p_s$ are the degrees of the singularities, then the sum of the angles
is $(p_1+2)\pi+\ldots+(p_s+2)\pi=(2k-2)\pi$, and thus that
\begin{displaymath}
p_1+\ldots+p_s+2s+2=2k.
\end{displaymath}
In the example of Figure~\ref{fig:suspension}, one has $s=4$, $p_1=p_2=p_3=p_4=-1$ and $k=3$.

The following lemma is the converse of Lemma~\ref{lem:Poincare_map}.
\begin{lemma} \label{lem:existence_suspension}
Let $T$ be a linear involution  without  connection.
Then there exists a measured foliation $(X,\Sigma,\FFF,\mu)$ without connection
and an admissible interval $I \subset X$ such that $T$ is conjugate to the Poincar\'e map of the foliation $\FFF$ on $I$.
\end{lemma}

\begin{proof}
We just use the reverse procedure as in the proof of 
Lemma~\ref{lem:Poincare_map}. For each subinterval $I_a$, we consider a 
rectangle $R_a = I_a \times [0,1]$. The vertical boundaries of the rectangles 
can be glued together to give a foliation. Note that there is no need to glue
the vertical sides of the rectangles by isometry since we are only interested
in the transverse measure $dx$.
\end{proof}
The pair $(\FFF,\mu,I)$ of a measured foliation and an admissible
interval associated with $T$ as above is called a \emph{suspension} of $T$.

\section{Natural codings}\label{sec:returnFund}

We now focus on return words of  linear foliations. 
Algebraic information  on  the set of return words (see Theorem~\ref{theoremReturns} below)   then will  follow from the remark that a section 
captures the geometry of the surface (see Lemma~\ref{lem:key_lemma}) and that the free group is geometrically seen as the fundamental group $\pi_1(X \backslash \Sigma)$.

%%%%%%%%%%%%%%%%%%%%%%%%%%%%%%%%
\subsection{Natural codings of linear involutions}\label{sectionNatural}
In this section, we introduce the  natural coding  of  a  linear
involution $T$. It is  obtained by  first coding the orbits under $T$ with respect to the partition provided
by the intervals $I_a$ ($a \in A \cup A^{-1}$), and then,  by  taking the  language of the  associated symbolic  dynamical system.

%%%%%%%%%%%%%%%%%%%%%%%
%\subsection{Natural coding }\label{subsectionNatural}
Let $T$ be a  linear involution on $I$, let $\hat{I}=I\times\{0,1\}$ and let
$\hat{O}$ be the set defined by Equation~\eqref{eqO}.
Given $z\in \hat{I}\setminus \hat{O}$, the \emph{infinite natural
  coding} of $T$ relative to $z$ is the infinite word
$\Sigma_T(z)=a_0a_1\ldots$
on the alphabet $A\cup A^{-1}$ defined by
\begin{displaymath}
a_n=a\quad\text{ if }\quad T^n(z)\in I_a.
\end{displaymath}
We first observe that the infinite word $\Sigma_T(z)$ is
reduced. Indeed, assume that $a_n=a$ and $a_{n+1}=a^{-1}$ with $a\in
A\cup A^{-1}$. Set $x=T^n(z)$ and $y=T(x)=T^{n+1}(z)$. Then $x\in I_a$
and $y\in I_{a^{-1}}$. But $y=\sigma_2(u)$ with $u=\sigma_1(x)$. Since $x\in
I_a$, we have $u\in I_{a^{-1}}$. This implies that
$y=\sigma_2(u)$ and $u$ belong to the same component of $\hat{I}$, a contradiction.

\begin{definition}[Natural coding]
Let $T$ be a  linear involution.
We   let  ${\mathcal L}(T)$ denote  the set of factors of the infinite natural codings
of $T$. We say that ${\mathcal L}(T)$ is 
the \emph{natural coding} of $T$.  \end{definition}

As  classically done in symbolic dynamics for  codings,    the  set ${\mathcal L}(T)$  can be easily described in terms of intervals
 associated with factors, obtained  by refining the coding partition.
\begin{lemma}
For a  nonempty word $w=a_0a_1\cdots a_{m-1}$ on $A\cup A^{-1}$, we define
$$
I_w=I_{a_0}\cap T^{-1}(I_{a_1})\cap\ldots\cap T^{-m+1}(I_{a_{m-1}}).
$$
By convention, $I_\varepsilon=\hat{I}$.
We have
$$
u\in {\mathcal L } (T)\Longleftrightarrow I_u\ne\emptyset.%\label{equationuinL}
$$
\end{lemma}
\begin{proof} 
For any $z\in\hat{I}\setminus\hat{O}$, one has $z\in I_u$
if and only if $u$ is a prefix of $\Sigma_T(z)$.

Each set $I_u$ is a (possibly empty) open interval. Indeed, this
is true if $u$ is a letter. Next, assume that $I_u$ is an open interval.
Note that
\begin{equation}
I_{au}=I_a\cap T^{-1}(I_u). \label{equationI}
\end{equation}

Then, by \eqref{equationI}, for $a\in A\cup A^{-1}$, we have 
$T(I_{au})=T(I_a)\cap I_u$ and thus $T(I_{au})$ is an open interval.
Since $I_{au}\subset I_a$, $T(I_{au})$ is the image of $I_{au}$ by a
continuous map and thus $I_{au}$ is also an open interval.

If $u$ is a factor of $\Sigma_T(z)$ for some 
$z\in\hat{I}\setminus\hat{O}$, then $T^n(z)\in I_u$ for some $n\ge
0$ and thus $I_u\ne\emptyset$. Conversely, if $I_u\ne\emptyset$, since $I_u$
is an open interval, it contains some $z\in \hat{I}\setminus\hat{O}$. 
Then $u$ is a prefix of $\Sigma_T(z)$ and thus $u\in {\mathcal L } (T)$.
\end{proof}

Observe that if  $T$ is nonorientable and without connection, then by Proposition~\ref{propBL}, ${\mathcal L } (T)$
is the set of factors of  $\Sigma_T(z)$ for any $z\in\hat{I}\setminus\hat{O}$, that is, 
the set of factors of $\Sigma_T(z)$ does not depend on $z$. 
Indeed,  if $I_u\ne\emptyset$, since the orbit of $z$ is dense
in $\hat{I}$, there
is an $n\ge 0$ such that $T^n(z)\in I_u$ and thus $u$ is a factor of $\Sigma_T(z)$.

\begin{proposition} \label{prop:inverse}
Let $T=\sigma_2\circ\sigma_1$ be a linear involution.
 For any nonempty word  $u\in {\mathcal L } (T)$, one has $I_{u^{-1}}=\sigma_1T^{|u|-1}(I_u)$.
Consequently the set ${\mathcal L } (T)$ is closed under taking inverses. It is thus a laminary set.
\end{proposition}
\begin{proof}
To prove the  assertion, we use an induction on the
length of $u$. The property holds for $|u|=1$ by definition of $\sigma_1$.
 Next, consider
 $u\in {\mathcal L } (T)$ and $a\in A\cup A^{-1}$ such that $ua\in {\mathcal L } (T)$. We assume by induction
hypothesis that $I_{u^{-1}}=\sigma_1T^{|u|-1}(I_u)$.

Since $T^{-1}=\sigma_1\circ\sigma_2$,
\begin{eqnarray*}
\sigma_1T^{|u|}(I_{ua})&=&\sigma_1T^{|u|}(I_u\cap T^{-|u|}(I_a))
=\sigma_1T^{|u|}(I_u)\cap \sigma_1(I_a)\\
&=&\sigma_1\sigma_2\sigma_1T^{|u|-1}(I_u)\cap \sigma_1(I_a)
=\sigma_1\sigma_2(I_{u^{-1}})\cap I_{a^{-1}}=I_{a^{-1}u^{-1}}
\end{eqnarray*}
where the last equality results from the application of
Equation~\eqref{equationI} to the word $a^{-1}u^{-1}$.

We  easily  deduce that   the set ${\mathcal L } (T)$ is closed under taking inverses. 
 Furthermore it  is a   factorial subset of the free group $F_A$. It  is  thus a laminary set.
\end{proof}

%%%%%%%%%

\begin{example}%[A  linear involution that is a fixed point of a morphism] 
Let $T$ be the  linear involution of Example \ref{exampleInvolution3}.
The set $S={\mathcal L } (T)$ can  actually be defined directly as the set of factors
of the substitution
\begin{displaymath}
f:a\mapsto cb^{-1},\quad b\mapsto c,\quad c\mapsto ab^{-1}
\end{displaymath}
which extends to an automorphism of the free group $F_A$.
The verification uses the Rauzy induction initially defined by
Rauzy and extended to linear involutions in~\cite{BoissyLanneau2009}.
The Rauzy induction applied to $T$ gives the linear involution $T'$
represented in Figure~\ref{RauzyFigure} on the left. It is the transformation
induced by $T$ on the interval obtained by erasing the smallest interval
on the right, namely $I_{a^{-1}}$.

The Rauzy induction applied on $T'$  is obtained by erasing the
smallest interval on the right, namely $I_{b^{-1}}$. It
gives a transformation $T''$
represented in Figure~\ref{RauzyFigure} on the right.

The transformation $T''$
 is the same as $T$ up to normalization of the length of the interval,
 exchange of the two components and the permutation (written in cycle form)
$\pi=(a\, c\, b\, a^{-1}\, c^{-1}\, b^{-1})$ (see Figure~\ref{RauzyFigure})
which sends $a$ to $c$, $c$ to $b$ and so on.
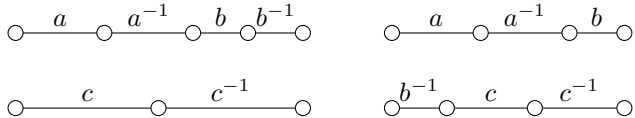
\begin{figure}[hbt]
\centering
\gasset{AHnb=0,Nadjust=wh}
\begin{picture}(100,20)
\put(0,0){
\begin{picture}(50,20)
\node(h0)(0,10){}\node(abar)(11.8,10){}\node(b)(23.6,10){}\node(bbar)(30.9,10){}\node(h1)(38.2,10){}
\node(b0)(0,0){}\node(cbar)(19,0){}\node(b1)(38.2,0){}

\drawedge(h0,abar){$a$}\drawedge(abar,b){$a^{-1}$}\drawedge(b,bbar){$b$}\drawedge(bbar,h1){$b^{-1}$}
\drawedge(b0,cbar){$c$}\drawedge(cbar,b1){$c^{-1}$}
\end{picture}
}
\put(50,0){
\begin{picture}(50,20)
\node(h0)(0,10){}\node(abar)(11.8,10){}\node(b)(23.6,10){}\node(h1)(30.9,10){}
\node(b0)(0,0){}\node(c)(7.3,0){}\node(cbar)(19,0){}\node(b1)(30.9,0){}

\drawedge(h0,abar){$a$}\drawedge(abar,b){$a^{-1}$}\drawedge(b,h1){$b$}
\drawedge(b0,c){$b^{-1}$}\drawedge(c,cbar){$c$}\drawedge(cbar,b1){$c^{-1}$}
\end{picture}
}
\end{picture}
\caption{The transforms $T'$ and $T''$ of $T$ by Rauzy induction.}\label{RauzyFigure}
\end{figure} 

Set $S={\mathcal L } (T)$, $S'={\mathcal L}(T')$ and $S''={\mathcal L}(T'')$ and let  $\Fact(X)$
denote  the set of factors of a set of
words $X$.
Since $T'$ is obtained from $T$ by a Rauzy induction,
there is an associated  automorphism $\tau'$
of the free group such that $S=\Fact(\tau'(S'))$.
One has actually $\tau:a\mapsto ab^{-1}, b\mapsto b,c\mapsto c$.
Similarly, one has $S'=\Fact(\tau''(S''))$ with
$\tau'':a\mapsto a,b\mapsto bc^{-1},c\mapsto c$. 
Set $\tau=\tau'\circ\tau''$. It is easy
to verify that $f=\tau\circ \pi^{-1}$. Since 
$S=\Fact(\tau(S''))=\Fact(\tau\pi^{-1}(S))=\Fact(f(S))$,
we obtain that $S$ is the set of factors of the fixpoint of $f$ 
as claimed above.
\end{example}

%%%%%%%%%%
\subsection{Orientability and uniform recurrence }
We gather  here  basic properties of the language ${\mathcal L } (T)$ of a linear involution.
We recall that  the notion of orientability for  a laminary  set  was introduced  in Section \ref{sectionPreliminaries}.
\begin{proposition}\label{propOrientable} 
Let $T$ be a linear involution.
If $T$ is orientable, then ${\mathcal L } (T)$ is orientable.
The converse is true  if $T$ has no connection.
\end{proposition}
\begin{proof}
Let $T$ be a linear involution and let $S={\mathcal L } (T)$.
Assume that $T$ is orientable. Set $S_+=\{u\in S\mid I_u\subset I\times\{0\}\}\cup\{\varepsilon\}$
and $S_-=\{u\in S\mid I_u\subset I\times\{1\}\}\cup\{\varepsilon\}$. Then $S=S_+\cup S_-$.
Since $T$ is orientable, we have $u\in S_+$ (resp. $u\in S_-$)
if and only if all letters
of $u$ are in $S_+$ (resp. in $S_-$). 
This shows that $S_+\cap S_-=\{\varepsilon\}$, that $S_+,S_-$ are factorial,
and that $u\in S_+$ if and only if $u^{-1}\in S_-$. Thus $S$ is orientable.

Conversely, assume that $T$  is nonorientable and has no connection.
 Let $a\in A$
be such that $I_a,I_{a^{-1}}\subset I\times\{0\}$. Since 
 $T$  is minimal by Proposition~\ref{propBL}, 
there is some $z\in I_a$ and $n>0$ such that $T^n(z)\in I_{a^{-1}}$.
Thus $S$ contains a word of the form $aua^{-1}$. This implies that
$S$ is nonorientable.
\end{proof}
The following statement can be easily deduced from the similar statement for interval exchange transformations (see~\cite[p. 392] {BertheRigo2010}).
\begin{proposition}\label{propUR}
Let $T$ be a  linear involution without connection. If $T$ is nonorientable,
then ${\mathcal L } (T)$ is uniformly recurrent. Otherwise, ${\mathcal L } (T)$ is uniformly
semi-recurrent.
\end{proposition}
\begin{proof}
Set $S={\mathcal L } (T)$.
Let $u\in S$ and let $N$ be the maximal return time to $I_u$
(this exists by Proposition~\ref{propBL}).
Thus for any $z\in \hat{I}$ such that $\rho_{I_u}(z)$ is finite, we have $\rho_{I_u}(z)\le N$. 
Let $w$ be a word of $S$ of length $N+|u|$ and let $z\in\hat{I}\setminus\hat{O}$
be such that $\Sigma_T(z)$ begins with $w$.

If $T$ is nonorientable, by Proposition~\ref{propBL}, it is minimal.
Thus there exists $n>0$ such that $T^n(z)\in I_u$. This implies that
$\rho_{I_u}(z)$ is finite and thus that $\rho_{I_u}(z)\le N$. This implies
in turn that $u$ is a factor of $w$. We conclude that $S$ is uniformly
recurrent.

If $T$ is orientable, then the restriction of $T$ to each component
of $\hat{I}$ is minimal. By Proposition~\ref{propOrientable}, $S$
is orientable. Thus $I_u$ and $I_{u^{-1}}$ cannot be included in the
same component of $\hat{I}$, since otherwise $S$ would contain
a word of the form $uvu^{-1}$,  and $S$ would be nonorientable.
Thus $I_w$ is in the same component as $I_u$ or $I_{u^{-1}}$, 
and we conclude as above that $u$ or $u^{-1}$ is a factor of $w$. This shows that
$S$ is uniformly semi-recurrent.
\end{proof}
%%%%%%%%%%%%%%%%%%

%%%%%%%%
\subsection{Return words and the even group}\label{subsec:return}

In this section, we first  introduce  odd and even words, and then discuss  various notions of return words.
%Let $T$ be a nonorientable linear involution on $I$ without connection and let
%$S={\mathcal L } (T)$ be its natural coding.

\begin{definition}[Even  group]
Let $T$ be a   linear involution on $I$ without connection. 

We say that a letter $a\in A$ is \emph{even} (with respect to $T$)
if $I_a$ and $I_{a^{-1}}$ belong to distinct components of $\hat{I}$ and \emph{odd},  otherwise.

A reduced word is said to be \emph{even} if it has an even number of odd letters and
said to be \emph{odd}, otherwise. 
In particular, if  $T$ is orientable,  all words are even.

The {\em even group}  is  the subgroup of the free group $F_A$  formed by the even words.
\end{definition}
Note that a word $w$ is even
if and only if for any $z\in I_w$, the points $z$ and $T^{|w|}(z)$
belong to the same component.   Since $\sigma_2 I_{w^{-1}}=T^{|w|}(I_w)$ according to  Proposition \ref{prop:inverse}, $w$ is even if and only if $I_{w}$ and $I_{w ^{-1}}$ belong to distinct components of $\hat{I}$. Hence  a word $w$ is even
if and only if  $ I_w$   and $T^{-|w|}I_w$
belong to the same component.  

If  $T$ is assumed to be nonorientable, the  even group is a  subgroup of index $2$
of $F_A$; it has thus   rank $2  \Card{A}  -1$ according to Schreier's formula.

\begin{example}\label{exampleEvenOdd}
 Let $T$ be the linear involution of Example \ref{exampleInvolution3}.
The letter $a$ is even and the letters $b,c$ are odd. The even group is generated by
the set $X=\{a,b\bar{a}c,b\bar{c},\bar{b}\bar{c},\bar{b}c\}$.
\end{example}

We now introduce several notions of return words.
Let $T$ be a linear involution on $I$ relative to the alphabet $A$
and let $S={\mathcal L } (T)$ be its natural coding.  Recall that $S$ is a factorial  subset of the free  group $F_A$.

For  a set $X\subset S$, a \emph{complete return word}
to $X$ is a word of $S$ which has a proper prefix in $X$ and
a proper suffix in $X$. 
 A \emph{complete first return word}    is a  complete return word  to $X$  that  has no internal factor in $X$.
 %The set   of simple complete return
%words to $X$
% is a bifix code.
If $S$ is uniformly recurrent (in particular, if  $T$ is  nonorientable and  without connection, by Proposition \ref{propUR}),   the set of  complete  first return
words to $X$  is finite for any finite
set $X$.

We now    focus on return words for  two  types of sets $X$, namely    sets reduced to   one word  or  symmetric sets of the form $\{ w, w^{-1}\}$.

By considering  the set $\{w\}$, one recovers  the classical notion of return word.
 For any $w\in S$, a \emph{first right return
word} to $w$ in $S$ is a word $u$ such that $wu$ is a  complete first  return
word to $\{w\}$. We denote by $\RR_S(w)$ the set of first right return
words to $w$ in $S$.  We define similarly first  left return words.

\begin{remark}Note that all elements of $\RR_S(x)$ are even.
Indeed, if $w\in\RR_S(x)$, we have $xw=vx$ for some $v\in S$. 
We assume  w.l.o.g. that  $x$ is odd and that $I_x\subset I \times \{0\}$. Take $z \in I_{xw}$. Then $T^{|x|} (z) \in  I \times \{1\}$  since $x$ is odd.
One has  $T^{|x|} (z) \in  I_w$. Hence $I_w   \subset  I \times \{1\}$. But  $T^{|w|} (I_w) \subset  T^{-|x|} I_x \subset  I \times \{1\}$ (again since $x$ is odd).
Hence  $T^{|w|} (I_w)$ and $I_w$ belong to the same component and $w$ is even. The other cases can be handled similarly.
\end{remark}

For $w\in S$, we also  consider  complete first  return words to the
set $X=\{w,w^{-1}\}$ in $S$.   We  let  $\CR_S(w)$ denote   this set and call its elements  the \emph{ complete  first return words to $\{w,w^{-1}\}$}.

In order to  provide  a connection between  return words  and  elements of a symmetric basis of the free group,   we need to introduce
a further  notion that plays the  role of  usual first return words in symbolic dynamics.
\begin{definition}Mixed  first return words]\label{def:mixed}
With a complete return word $u$ to  the set $\{w,w^{-1}\}$, we associate a word $N(u)$  as follows: if $u$ has  $w$ as prefix, we erase it
and if $u$ has a suffix $w^{-1}$, we also erase it.  Such a word is called a   {\em   mixed return word}.

The words $N(u)$ for $u$ complete first  return word  to  $\{w,w^{-1}\}$ are called {\em   mixed first  return words}.
We  let  $\MR_S(w)$ denote   this  set. \end{definition}
Note that the two operations  described above  can be made in any order since
$w$ and $w^{-1}$ cannot overlap.
 Note  also that $\MR_S(w)$ is symmetric and that $w^{-1}\MR_S(w)w=\MR_S(w^{-1})$.

If $T$ is  orientable,  then  $\MR_S(w)$  is equal to the union of the  set   of first right return
words to $w$ with   the  set   of first left  return
words to $w^{-1}.$

Observe  that any uniformly recurrent 
biinfinite word $x$ such that $F(x)=S$ can be uniquely written 
as a concatenation of  mixed first  return words (see Figure~\ref{figureFact}).
Note also  that successive occurrences of $w$ may overlap but that
successive occurrences of $w$ and $w^{-1}$ cannot.
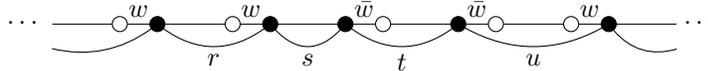
\begin{figure}[hbt]
\centering\gasset{Nadjust=wh,AHnb=0}
\begin{picture}(100,10)(0,-3)
\put(0,0){$\ldots$}
\node[Nframe=n](0)(5,0){}\node[Nframe=n](1b)(5,-3){}
\node(1)(15,0){}\node[fillcolor=black](2)(20,0){}
\node(3)(30,0){}\node[fillcolor=black](4)(35,0){}
\node[fillcolor=black](5)(45,0){}\node(6)(50,0){}\node[fillcolor=black](7)(60,0){}
\node(8)(65,0){}\node(9)(75,0){}\node[fillcolor=black](10)(80,0){}\node[Nframe=n](11)(90,0){}
\node[Nframe=n](11b)(90,-3){}

\drawedge(0,1){}\drawedge(1,2){$w$}\drawedge[curvedepth=-2](1b,2){}
\drawedge(2,3){}\drawedge(3,4){$w$}\drawedge[curvedepth=-3,ELside=r](2,4){$r$}
\drawedge(4,5){}\drawedge(5,6){$\bar{w}$}\drawedge[curvedepth=-3,ELside=r](5,7){$t$}
\drawedge[curvedepth=-3,ELside=r](4,5){$s$}\drawedge(6,7){}
\drawedge(7,8){$\bar{w}$}\drawedge(8,9){}\drawedge(9,10){$w$}
\drawedge[curvedepth=-3,ELside=r](7,10){$u$}\drawedge[curvedepth=-2](10,11b){}
\drawedge(10,11){}\put(90,0){$\ldots$}
\end{picture}
\caption{A uniformly recurrent infinite word factorized as an infinite product
$\cdots rstu\cdots$ of  mixed  first return words to $w$.}\label{figureFact}
\end{figure}

\begin{example}\label{exampleInvolution3bis}
Let $T$ be the  linear involution of Example~\ref{exampleInvolution3}.
We have 
\begin{eqnarray*}
\CR_S(a)&=&\{a\bar{b}cb\bar{a},a\bar{b}cb\bar{c}a,\bar{a}c\bar{b}\bar{c}a,a\bar{b}\bar{c}b\bar{a},\bar{a}cb\bar{c}a,\bar{a}c\bar{b}\bar{c}b\bar{a}\}\\
\CR_S(b)&=&\{b\bar{a}cb,b\bar{a}c\bar{b},b\bar{c}a\bar{b},\bar{b}cb,
\bar{b}\bar{c}a\bar{b},\bar{b}\bar{c}b\},\\
\CR_S(c)&=&\{cb\bar{a}c,cb\bar{c},c\bar{b}\bar{c},\bar{c}a\bar{b}c,
\bar{c}a\bar{b}\bar{c},\bar{c}b\bar{a}c\}
\end{eqnarray*}
and 
\begin{eqnarray*}
\MR_S(a)&=&\{\bar{b}cb,\bar{b}cb\bar{c}a,\bar{a}c\bar{b}\bar{c}a
,\bar{b}\bar{c}b,\bar{a}cb\bar{c}a,\bar{a}c\bar{b}\bar{c}b\}\\
\MR_S(b)&=&\{\bar{a}cb,\bar{a}c,\bar{c}a,\bar{b}cb,
\bar{b}\bar{c}a,\bar{b}\bar{c}b\},\\
\MR_S(c)&=&\{b\bar{a}c,b,\bar{b},\bar{c}a\bar{b}c,
\bar{c}a\bar{b},\bar{c}b\bar{a}c\}.
\end{eqnarray*}
\end{example}

The reason for introducing the notion  of mixed return words (see  Definition \ref{def:mixed})
comes from the fact that %, when $S$ is the natural coding
%of a linear involution, 
we are interested in the transformation induced
on $I_w\cup\sigma_2(I_w)$, according to Section \ref{sectionGeometric}.
The natural coding of a point in $I_w$ begins
with $w$ while the natural coding of a point $z$ in $\sigma_2(I_w)$   is preceded 
by  $w^{-1}$ in the sense that the natural coding of  $T^{-|w|}(z)$ 
begins with $w^{-1}$. To be more precise,  the convention  chosen for  the transformation $N$  corresponds to the induction
on   $I_{w^{-1}} \cup \sigma_2(I_{w^{-1}})$, such  as shown with the following lemma.
 Recall that the notation  $\rho_X$  stands for the return time to $X$.
\begin{lemma} \label{lem:return_word_code_induced_map}
Let $T$ be a linear involution with no connection and $w$ a nonempty word in its  natural coding  ${\mathcal L } (T)$.
Let $K_w = I_{w^{-1}} \cup \sigma_2(I_{w^{-1}})$. Then the set of  mixed first  return words to
$w$ are exactly the prefixes of length $\rho_{K_w}(z)$ of the infinite
natural coding of points  $z\in K_w$.
\end{lemma}

\begin{proof}
Let  $u$ be the prefix of length $\rho_{K_w}(z)$ of $\Sigma_T(z)$
for  some $z\in K_w$.
Let us first recall  that $\sigma_2(I_{w^{-1}}) = T^{|w|}(I_{w})$ (Proposition \ref{prop:inverse}).
Assume first that the length of $u$ is larger than  or equal to the length of $w$. If $z \in I_{w^{-1}}$, then
$u$ starts with $w^{-1}$ while if $z \in \sigma_2(I_{w^{-1}})$ then $w u$ is in ${\mathcal L } (T)$.
Similarly, if $T^{|u|}(z) \in I_{w^{-1}}$ then $u w^{-1}$ is in ${\mathcal L } (T)$ while if $T^{|u|}(z) \in \sigma_2(I_{w^{-1}})$ then
$u$ ends with $w$. In all four possible cases, $u$, $w u$, $u w^{-1}$ and
$w u w ^{-1}$ are in ${\mathcal L } (T)$.

Let
\[
p = \left\{ \begin{array}{ll}
\epsilon & \text{if $z \in I_{w^{-1}}$},\\
  w & \text{if $z \in \sigma_2( I_{w^{-1}})$},
\end{array} \right.
\qquad
\text{and}
\qquad
s = \left \{ \begin{array}{ll}
w ^{-1}& \text{if $T^{|u|}(z) \in I_{w^{-1}}$}, \\
\epsilon & \text{if $T^{|u|}(z) \in  \sigma_2(I_{w^{-1}})$.}
\end{array} \right.
\]
Since $I_{w^{-1}}$ and $\sigma_2(I_{w^{-1}})$  are included into two distinct components, there is no cancellation in the product $p u s$. Moreover, $|p u s| \geq |u|$ and
hence $p u s$ starts and ends with an occurrence of $w$ or $w^{-1}$. It is thus  a complete
return word to $\{w,w^{-1}\}$. Furthermore one has $N(pus)=u$.

\smallskip

Let conversely $u$ be a mixed first  return word to $w$ and let $u'$ be the  complete first   
return word such that $u=N(u')$.  Write $u'=pus$. 
Assume first that $u'=wu$. Then $wu$
ends with $w$. For any point $y\in I_{u'}$, set $x=T^{|w]}(y)$.
Then  $x \in T^{|w]}  I_{w^{}}=  \sigma_2(I_{w^{-1}})$,  $x\in I_u$, and thus    $ T^{|u|} x  \in   \sigma_2(I_{w^{-1}})$ and $\rho_{K_w}(x)=|w|$.
Hence $u$ is the prefix of length $\rho_{J_w}(x)$ of $\Sigma_T(x)$.
The proof in the three other cases is similar.
\end{proof}

We end this section by introducing a further variation around return words, adapted to subgroups  of the free group (the  interest of this notion will be highlighted by  Theorem \ref{theoremGroupCode} below).
\begin{definition}[Prime words]
Let $G$
be a subgroup of the free group $F_A$. Let $S$ be a  laminary  set.
The \emph{prime words}  in $S$ with respect to $G$  are the 
nonempty words
in $G \cap S$  without a proper nonempty prefix  in $G \cap S$.
\end{definition}

\begin{example}
Let $T$ be the linear involution of Example \ref{exampleInvolution3}.
The set of prime words with respect to the even group is the set $X\cup X^{-1}$
where $X$ is as in Example~\ref{exampleEvenOdd}.
\end{example}

\section{Return words and  fundamental group}\label{sec:return}
%%%%%%%%%%%%%%%%%%%%%%%%%
%\subsection{The free groups as }\label{subsectionFundamentalGroup}

We now   interpret  the notions
of `return words' we have seen so far   (to a word, or with respect to  a subgroup via the notion of prime words)  
in geometrical terms.

We consider  a punctured surface $(X,\Sigma)$. Fixing a base point $x_0$, 
recall that the \emph{fundamental group} $\pi_1(X \backslash \Sigma, x_0)$ 
is the set of 
equivalence classes of loops in $X \backslash \Sigma$ based at $x_0$ 
up to homotopy. One ingredient of  our main  results    (Theorem \ref{theoremReturns} and Theorem \ref{theoremGroupCode})  is that with  each 
admissible interval  for  the foliation (in the sense  of Definition \ref{definition1})  is associated a \smc  of the fundamental group as we shall see below.
Furthermore, the fundamental group   is    a free group.

Let $(X,\Sigma,\FFF,\mu)$ be a measured foliation and assume that $\Sigma$ 
is nonempty. Let $I$ be an admissible interval and let $x_0$ be any point of $I$. By Lemma~\ref{lem:Poincare_map}, the domain $I \times \{0,1\}$ of the Poincar\'e map $T$
is cut into $2k$ subintervals by the first return map. With each subinterval $I_a$ we associate an element of $\pi_1(X \backslash \Sigma, x_0)$ as follows. Let $x$ be a point in that subinterval, we consider the loop $\gamma(x)$ which is the concatenation of
\begin{itemize}
\item the  segment in $I$ that joins $x_0$ to $x$,
\item the piece of leaf that joins $x$ to $x' = T(x)$,
\item the  segment in $I$ that joins $x'$ to $x_0$.
\end{itemize}
The homotopy class of $\gamma(x)$ only depends on the subinterval to which $x$ belongs. We let   $\Gamma(X,I,x_0)$ denote   the set of  equivalence classes of loops in $\pi_1(X \backslash \Sigma ,x_0)$ obtained by that process. The following lemma shows in
particular that there are $2k$ classes.

\begin{lemma} \label{lem:key_lemma}
Let $(X,\Sigma,\FFF,\mu)$ be a measured foliation with total angle $(2k-2) \pi$.
Then, if $\Sigma$ is nonempty, the fundamental group of $X \backslash \Sigma$ is a free group on $k$ generators.
Moreover for any admissible interval $I$ in $X$ and any $x_0\in I$, the set $\Gamma(X,I,x_0)$ is a \smc of $\pi_1(X \backslash \Sigma, x_0)$.
\end{lemma}

\begin{proof}
Let $I$ be an admissible interval. We consider the $k$ loops obtained from the above construction. With an homotopy fixing $x_0$, one can easily realize the loops in such way that the only common point between any two is $x_0$. We let  $Y \subset X \backslash \Sigma$ denote   this set of $d$ loops. Now we show that the punctured surface $X \backslash \Sigma$ is homotopic to $Y$. We may decompose the surface $X \backslash \Sigma$ into zippered rectangles as in Lemmas~\ref{lem:Poincare_map} and~\ref{lem:existence_suspension}: we cut the surface along each singular leaf, from the singularities until the first time it hits the interior of $I$. In each rectangle there is exactly one loop passing through. It is easy to see that by a continuous deformation we can shrink each rectangle to that loop. In other words we build a homotopy to $Y$.

Now $Y$ is a connected sum of $k$ loops (also called a rose) and its fundamental group is a free group of rank $k$ generated by each curve that goes once through a loop.
\end{proof}

\subsection{Return words and bases of the free group}
We now have gathered all what  was needed to   deduce  algebraic properties of mixed first   return words. 

Let $T:I\times\{0,1\}\rightarrow I\times\{0,1\} $ be a linear involution on $A$
and let $S={\mathcal L } (T)$. 
We have  introduced with Definition~\ref{definition1} the notion of  an admissible interval
$I\subset X$ with respect to a measured foliation $(X,\Sigma,\FFF,\mu)$.
We can formulate directly a similar definition for an open interval $J\subset I$
with respect to a linear involution $T$ defined on $I$  as follows.
\begin{definition}[Admissible  interval]
Let $T$ be a linear involution  without connection  defined on the interval $I$.
 The open interval $J=]u,v[$ with  $J \subset I$ 
is {\em admissible} with respect to $T$ if for each of its two endpoints $x=u,v$, there is 
\begin{enumerate}
\item[(i)] either  a singularity $z$ of $T^{-1}$ such that $x=T^n(z)$
and $T^k(z)\notin J$ for $0\le k\le n$,
\item[(ii)] or   a singularity $z$ of $T$ such that $z=T^{n}(x)$
and $T^k(x)\notin J$ for $0\le k\le n$.
\end{enumerate}
\end{definition}
The term `admissible' was introduced originally by G. Rauzy \cite{Rauzy1979} for interval exchanges.

It is clear that if $J$ is admissible with respect to $T$, then it is
admissible with respect to any suspension $(\FFF,\mu,I)$ of $T$.
Hence,  for any  admissible interval of $I$  with respect to $T$,
 the transformation induced on $I$  is a
$k$-linear involution  without connection, according to Lemma \ref{lem:Poincare_map}.
Furthermore,  for any  admissible interval
of $I$, the Poincar\'e map  of the foliation  is the  Poincar\'e map of the linear involution  on the   union  $I\cup\sigma_2(I)$.

The following result  is proved in~\cite{DolcePerrin2015} for interval exchange
transformations. The proof for linear involutions is
the same. Recall that the   intervals $I_w$, $w\in S$,   are  defined in Section \ref{sectionNatural}.
\begin{proposition}\label{propIwAdmissible}
Let $T$ be a linear involution without connection on $I$. The interval $I_w$, seen as a subinterval of $I$,
is admissible with respect to $T$.
\end{proposition}

We now can state our  main result concerning return words.

 \begin{theorem}\label{theoremReturns}
Let $S$ be the natural coding of a   linear involution without connection on the alphabet $A$. 
For any $w\in S$, the set of  mixed  first return words to $w$ is a symmetric basis
of $F_A$.
\end{theorem}
\begin{proofof}{of Theorem~\ref{theoremReturns}}
Let $T$ be a 
linear involution without connection
on the alphabet $A$.
By Lemma~\ref{lem:existence_suspension}, there exists a measured foliation $(X,\Sigma,\FFF,\mu)$ and an admissible interval $I\subset X$ such that $T$ is conjugate to the Poincar\'e map of $\FFF$ on $I$.
Let $w$ be a nonempty word of the natural coding $S={\mathcal L } (T)$. 
By Proposition~\ref{propIwAdmissible}, the subinterval $I_w$ is    admissible    for  the linear involution $T$. 
Let $x_0$ be a point in $I_w$.
We have a natural identification $F_A \rightarrow \pi_1(X \backslash \Sigma, x_0)$ given by Lemma~\ref{lem:key_lemma}. Since $I_w$ is admissible, using Lemma~\ref{lem:return_word_code_induced_map}, the same construction provides an identification of the subgroup generated by the mixed  first  return words,
$\Gamma(X,I_w,x_0)$ and $\pi_1(X \backslash \Sigma, x_0)$. 
This shows that the set of  mixed  first return words is a \smc of $F_A$.
\end{proofof}

Theorem \ref{theoremReturns} thus  provides bases of the free group  within a given natural coding 
by taking    mixed  first return  words with respect to a given factor $w$.

\begin{example}
The set   of  $\MR_S(c)$ in Example \ref{exampleInvolution3bis} provides a  symmetric basis of the free group, whereas 
$\CR_S(c)$   is  not a symmetric basis of the free group.

\end{example}

One  also deduces the following  cardinality result, which is the counterpart for linear involutions of Theorem 3.6
in~\cite{BertheDeFeliceDolceLeroyPerrinReutenauerRindone2013a},  that holds for tree sets,  by  noticing that the set of mixed first  return words $\MR_S(w)$ has the same cardinality as  the set of
 complete first  return words $\CR_S(w)$.

\begin{corollary} Let $T$ be a 
linear involution without connection
on the alphabet $A$. For any $w\in {\mathcal L } (T)$, the set of  complete  first 
return  words to $\{w,w^{-1}\}$ has $2 \Card(A)$ elements. 
\end{corollary}

\subsection{Prime words and coverings}\label{subsec:coverings}

We now prove   an  analogue of  Theorem \ref{theoremReturns} for prime words with respect to a  subgroup of the free group. This will be Theorem \ref{theoremGroupCode}  below.  We will first  consider  surface coverings that are in correspondence with subgroups of $\pi_1(X \backslash \Sigma)$. From this correspondence, we will  obtain  a  proof of Theorem~\ref{theoremGroupCode}. 

Let us first  quickly recall the Galois correspondence of coverings.
Let $X$ be a compact connected surface and $\Sigma$ a finite set of points.
A \emph{covering} of $X$ of degree $d$ is a compact connected surface $Y$ with a continuous map $f: Y \rightarrow X$ such that for each $x \in X \backslash \Sigma$ there exists a connected neighborhood $U$ of $x$ such that $f^{-1}(U)$ is a disjoint union of $d$ open sets $f^{-1}(U) = U_1 \cup U_2 \cup \ldots \cup U_d$ such that for each $i \in \{1,\ldots,d\}$, $f: U_i \rightarrow U$ is a homeomorphism. In our case, we consider more generally a \emph{ramified covering} with ramifications contained in $\Sigma$. For points $x \in X \backslash \Sigma$ we keep the same condition, but for points $x \in \Sigma$ we allow the preimage to be a union of $m \leq d$ open sets $U_1 \cup U_2 \cup \ldots U_m$ such that $f$ restricted to $U_i$ is of the form $z \mapsto z^{p_i}$ for some $p_i \geq 0$ from the unit disc in $\CC$ to itself. One can show that $p_1 + p_2 + \ldots + p_m = d$. In other words, the degree is constant if we count multiplicities.

Two coverings $f:Y \rightarrow X$ and $f': Y' \rightarrow X$ are \emph{equivalent} if there exists an homeomorphism $g:Y \rightarrow Y'$ such that $f = f' \circ g$.

If $\gamma$ is a loop in $Y$ then $f(\gamma)$ is a loop in $X$. Hence, for any $y_0 \in Y$ we get a map $f_*: \pi_1(Y \backslash f^{-1}(\Sigma), y_0) \rightarrow \pi_1(X \backslash \Sigma, f(y_0))$. The map $f_*$ is injective and its image is of finite index in $\pi_1(X \backslash \Sigma, f(y_0))$.

The following result establishes a Galois correspondence between coverings of finite degree of $X$
ramified over $\Sigma$ and subgroups of $\pi_1(X\setminus\Sigma)$.
For a proof, see~\cite{Hatcher2002} or \cite{Forster1991}.
\begin{theorem} \label{thm:Galois_correspondence}
Let $X$ be a compact connected surface and let $\Sigma \subset X$ be a finite set.
Let $Y$ be   a covering of $X$ of degree $d$.
Then, the map $(f: Y \rightarrow X) \mapsto f_*(\pi_1(Y\setminus f^{-1}(\Sigma)))$ induces a bijection between
equivalence classes of coverings of degree $d$ ramified over $\Sigma$ and conjugacy
classes of subgroups of $\pi_1(X \backslash \Sigma)$ of index $d$.
\end{theorem}

\begin{example}\label{exampleEvenGroup}
Let $T$ be the linear involution  of Example~\ref{exampleInvolution3}.
It is  without connection and  nonorientable,  the group of even words is thus  a subgroup of index $d=2$.
The covering of degree $2$ of its suspension
 associated with the group of even words  is the orientation covering
 of the foliation. 
 
\begin{figure}[hbt]
\begin{center}\includegraphics{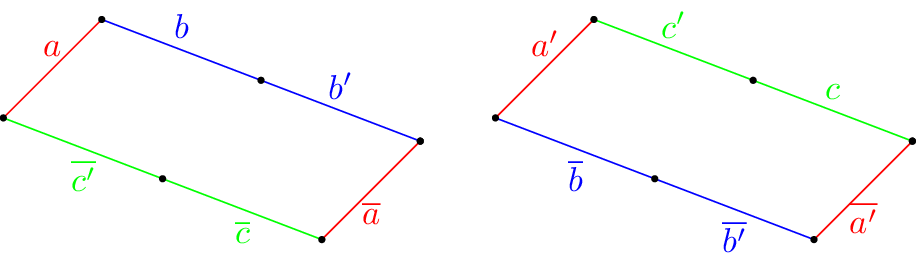}
%\begin{center}\includegraphics{suspension2.pdf}
\end{center}
\caption{The orientation covering of the suspension of Figure~\ref{fig:suspension}. The choice of letters is made in order that only positive letters or negative letters appear in the coding of an orbit.}
\label{fig:suspension_orientation}
\end{figure}

One can see on Figure~\ref{fig:suspension_orientation} that the obtained foliation is orientable. 
The result is actually a torus and its coding  yields Sturmian words.  Indeed, one  way to obtain   the orientation covering is to duplicate the alphabet and to work on $  (A \cup A') \cup (A\cup  A')^{-1}$. With each word are associated two  lifted words:
 the first one  is  obtained by replacing   the positive letters by elements of  $A$ and  negative letters by elements of  $A'$,  and
the second one   is obtained
by replacing  the positive  letters  by letters of   $(A')^{-1}$ and the negative ones by  elements of  $A^{-1}$.
The language  on $(A\cup A') \cup  (A\cup A')^{-1}$ that is obtained  in this way is  orientable.
As an illustration, the word    $c^{-1}  a  b^{-1 } c^{-1 } b  a^{-1 } c  $ belongs to  the natural coding  of  $T$   (see Figure \ref{fig:suspension}).
It admits  two  lifts that  code  orbits for the suspension  depicted  in   Figure \ref{fig:suspension_orientation}, namely
$  c'  a  b'  c'  b  a'  c' $
and
$ c^{-1}  (a')^{-1 } b^{-1}  c^{-1}  (b')^{-1}  a^{-1 } c^{-1 } .$
The word  $  c'  a  b'  c'  b  a'  c' $   belongs to the  natural coding  of   the interval exchange depicted below.  Even letters allow  one to stay in the same
half of this new interval exchange.

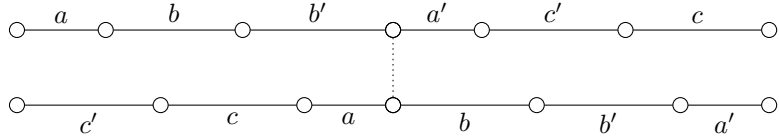
\begin{figure}[hbt]
\centering
\gasset{Nh=2,Nw=2,ExtNL=y,NLdist=2,AHnb=0,ELside=r}
\begin{picture}(100,10)
\node(0h)(0,10){}
\node(1h)(11.8,10){}
\node(2h)(30.0,10){}
\node(3h)(50,10){}
\drawedge[ELside=l](0h,1h){$a$}
\drawedge[ELside=l](1h,2h){$b$}
\drawedge[ELside=l](2h,3h){$b'$}

\node(0bx)(50,10){}
\node(1bx)(61.8,10){}
\node(2bx)(80.9,10){}
\node(3bx)(100,10){}
\drawedge[ELside=l](0bx,1bx){${a'}$}
\drawedge[ELside=l](1bx,2bx){${c'}$}
\drawedge[ELside=l](2bx,3bx){$c$}

\node(0b)(0,0){}
\node(1b)(19.1,0){}
\node(2b)(38.2,0){}
\node(3b)(50,0){}
\drawedge(0b,1b){${c'}$}
\drawedge(1b,2b){${c}$}
\drawedge(2b,3b){${a}$}

\node(0hx)(50,0){}
\node(1hx)(69.1,0){}
\node(2hx)(88.2,0){}
\node(3hx)(100,0){}
\drawedge(0hx,1hx){${b}$}
\drawedge(1hx,2hx){${b'}$}
\drawedge(2hx,3hx){${a'}$}

\drawedge[AHnb=0,dash={0.2 0.5}0](3h,3b){}
\end{picture}
\caption{Interval exchange corresponding  to the orientation covering.}
\end{figure}

\end{example}

The following statement gives a remarkable
property of the set of prime words with respect to a subgroup of finite index.

\begin{theorem}\label{theoremGroupCode}
Let $T$ be a  linear involution on $A$ without connection and let $S={\mathcal L } (T)$.
For any subgroup $G$ of finite index of the free group $F_A$, the set of  prime words in $S$ with respect to $G$
is a symmetric basis of $G$.
\end{theorem}

\begin{proofof}{of Theorem~\ref{theoremGroupCode}}
Let $T:\hat{I}\rightarrow \hat{I}$ be a   linear involution on 
the alphabet $A$ without connection. By Lemma~\ref{lem:existence_suspension},
there exists a measured foliation $(X,\Sigma,\FFF,\mu)$ and an admissible interval $I\subset X$ such that $T$ is conjugate to the Poincar\'e map of $\FFF$ on $I$.  By Lemma~\ref{lem:key_lemma},
there is an identification $F_A\rightarrow \pi_1(X\setminus\Sigma,x_0)$
for any $x_0\in I$.

Let $G$ be a subgroup of $F_A$ of index $d$.
By Theorem~\ref{thm:Galois_correspondence}, there is a covering 
$f:\tilde{X}\rightarrow X$
of degree $d$ ramified over $\Sigma$ such that $G$ is identified with
$\pi_1(\tilde{X}\setminus f^{-1}(\Sigma))$, i.e., 
 $f_*(\pi_1(\tilde{X}\setminus\Sigma))=G$.

The preimage $\tilde{I}$ of the interval $I$ in $\widetilde{X}$ is made of $d$ copies of $I$.
We can also lift the measured foliation to $\widetilde{X}$ and describe the Poincar\'e map
of this measure foliation on $\widetilde{I}$.
Indeed, let $\widetilde{I} = \hat{I}  \times Q$ where $Q$ is the set of right cosets of $G$ in $F_A$.
For a point $x \in \hat{I}$ we denote by $a(x)$ the element of $A \cup A^{-1}$ 
such that $x \in I_{a(x)}$.
We define
\[
\widetilde{T}(x, Gw) = (Tx, G w a(x)).
\]
Then $\widetilde{T}$ is the Poincar\'e map  of the lift of $(\FFF,\mu)$ to $\widetilde{X}$ on $\widetilde{I}$.

Now, consider the induced map of $\widetilde{T}$ on the interval $\hat{I} \times \{G\}$
where $\{G\}$ denotes the set reduced to the coset $G$.
For a point $x \in \hat{I}$ we denote by $\rho(x)$ the least $n \geq 1$ 
such that $\widetilde{T}^n(x,G) \in \hat{I} \times \{G\}$.

 The natural coding of a finite orbit $\{x,T(x),\ldots,T^{n-1}(x)\}$
is defined as  the word  $\Sigma_T^{(n)}(x)$ $=a_0a_1\cdots a_{n-1}$ such that $T^i(x)\in I_{a_i}$ for $0\le i<n$. Thus it is
is the prefix of length $n$ of the infinite natural coding $\Sigma_T(x)$ of $T$ relative to $x$. 

We fix a basepoint $\tilde{x}_0$ in $\tilde{X}$ and
for a point $x\in \hat{I}$, we denote by $\tilde{\gamma}(x)$ the loop from $\tilde{x}_0$  to itself which
passes by $x,T(x),\cdots,T^{\rho(x)-1}(x)$ as in the previous section.

It is easy to verify that the map $\tilde{\gamma}(x)\mapsto \Sigma_T^{\rho(x)}(x)$ for $x\in \hat{I}$
is a bijection from $\Gamma(\tilde{X}\setminus f^{-1}(\Sigma),\tilde{I},\tilde{x}_0)$ onto  the set of  prime words  with respect to  $G$ 
which extends to an isomorphism from $\pi_1(\tilde{X}\setminus\Sigma)$ onto $G$.

 By Lemma~\ref{lem:key_lemma}, the set $\Gamma(Y\setminus f^{-1}(\Sigma),\tilde{I}\times\{G\})$
is a symmetric basis of $G$.
We thus deduce that the set of prime words  with respect   to $G$ is a symmetric basis of $G$.
\end{proofof}

\begin{corollary}
Let $T$ be a linear involution without connection. Let $w$ be a word of its  natural coding  ${\mathcal L } (T)$.
The set of first right return words  to $w$  is a basis of the  even group.
\end{corollary}

\begin{proof}
We  assume  w.l.o.g.  that $I_w \subset  I \times \{0\}$.
 We consider the induced  map of $T$ on $I \times \{0\}$.
 It is an orientable   linear involution  without connection, that is,   an interval exchange with  flip(s), 
with intervals provided by   the  prime words of the   even group that belong to $S_+$, with the notation of Proposition \ref{propOrientable}.
 Furthermore, in the orientable case, the set of complete   first return words  $\MR(w)$   is made of the   first right return words  to $w$   
with the  first left return word to $w^{-1}$. The conclusion comes from the fact that  prime words  of the even group  that  are in $  S_+$ are the    first right return words  to $w$.
\end{proof}

We illustrate Theorem~\ref{theoremGroupCode} with the following 
interesting example.
\begin{example}\label{ex:returnbis}
Let $T$ be as in Example~\ref{exampleInvolution3}
and let $S={\mathcal L } (T)$. Let $G$ be the group of even words in $F_A$.
It is a subgroup of index $2$. The set of  prime words with respect  to $G$ in $S$
is the set $Y=X\cup X^{-1}$ with
\begin{displaymath}
X=\{a,ba^{-1}c,bc^{-1},b^{-1}c^{-1},b^{-1}c\}.
\end{displaymath}
Actually, the transformation induced by $T$ on the set $I\times \{0\}$
(the upper part of $\hat{I}$ in Figure~\ref{figureLinear3}) is the
interval echange transformation represented in Figure~\ref{figureRetourHaut}.
Its upper intervals are the $I_x$ for $x\in X$.
\begin{figure}[hbt]
\centering
\gasset{AHnb=0,Nadjust=wh}
\begin{picture}(100,20)
\node(h0)(0,10){}\node(babarc)(23.6,10){}\node(bcbar)(47.2,10){}
\node(bbarcbar)(61.8,10){}\node(bbarc)(85.4,10){}\node(h1)(100,10){}
\node(b0)(0,0){}\node(cb)(14.1,0){}\node(cbarb)(38.2,0){}\node(cbarabbar)(52.6,0){}
\node(abar)(76.4,0){}\node(b1)(100,0){}

\drawedge(h0,babarc){$a$}\drawedge(babarc,bcbar){$ba^{-1}c$}
\drawedge(bcbar,bbarcbar){$bc^{-1}$}\drawedge(bbarcbar,bbarc){$b^{-1}c^{-1}$}
\drawedge(bbarc,h1){$b^{-1}c$}
\drawedge(b0,cb){$cb^{-1}$}\drawedge(cb,cbarb){$cb$}\drawedge(cbarb,cbarabbar){$c^{-1}b$}
\drawedge(cbarabbar,abar){$c^{-1}ab^{-1}$}
\drawedge(abar,b1){$a^{-1}$}
\end{picture}
\caption{The transformation induced on the upper level.}\label{figureRetourHaut}
\end{figure}
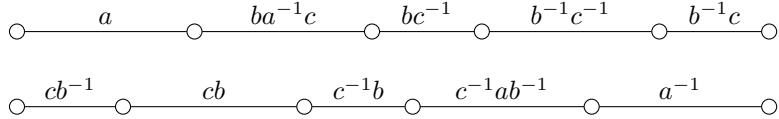
This corresponds to the fact that the words of $X$ correspond to
the first returns to $I\times\{0\}$ while the words of $X^{-1}$ correspond
to the first returns to $I\times\{1\}$.

Furthermore, one  may check directly  that the set  $X=\{a,ba^{-1}c,bc^{-1},b^{-1}c^{-1},b^{-1}c\}$ is a basis of a  subgroup of index $2$, in agreement with Theorem \ref{theoremGroupCode}. 
\end{example}

\bibliographystyle{plain}
\bibliography{involutions}
\end{document}